\documentclass[11pt]{amsart}
\usepackage{amsmath,amsthm,amssymb,graphicx,epsfig}
\usepackage[all]{xy}
\theoremstyle{plain}
\newtheorem{prop}{Proposition}

\newtheorem{theo}[prop]{Theorem}
\newtheorem{coro}[prop]{Corollary}

\newtheorem{lemm}[prop]{Lemma}

\newtheorem*{propo}{Proposition}  
\newtheorem*{exam}{Example}  

\theoremstyle{definition}

\newtheorem{rema}[prop]{Remark}
\newtheorem*{remar}{Remark}

\newcommand{\ra}{\rightarrow}

\newcommand{\C}{{\mathbb C}}

\newcommand{\Q}{{\mathbb Q}}

\newcommand{\cC}{{\mathcal C}}

\newcommand{\cE}{{\mathcal E}}
\newcommand{\cF}{{\mathcal F}}

\newcommand{\cK}{{\mathcal K}}

\newcommand{\cM}{{\mathcal M}}

\newcommand{\cO}{{\mathcal O}}
\newcommand{\cP}{{\mathcal P}}

\newcommand{\cT}{{\mathcal T}}

\newcommand{\oM}{{\overline{M}}}
\newcommand{\ocM}{{\overline{\mathcal M}}}

\newcommand{\ocP}{{\overline{\mathcal P}}}

\newcommand{\fS}{{\mathfrak S}}

\newcommand{\bP}{{\mathbb P}}
\newcommand{\bQ}{{\mathbb Q}}

\newcommand{\bZ}{{\mathbb Z}}

\newcommand{\Bl}{\mathrm{Bl}}

\newcommand{\Gr}{\mathrm{Gr}}
\newcommand{\Fl}{\mathrm{Fl}}
\newcommand{\OGr}{\mathrm{OG}}
\newcommand{\OG}{\mathrm{OG}}

\newcommand{\ev}{\mathrm{ev}}
\newcommand{\pt}{\mathrm{pt}}

\newcommand{\rank}{\mathrm{rank}}

\makeatother
\makeatletter

\author{Brendan Hassett}
\address{Department of Mathematics\\
Rice University, MS 136 \\
Houston, TX  77251-1892 \\
USA}
\email{hassett@rice.edu}

\author{Yuri Tschinkel}
\address{Courant Institute\\
                New York University \\
                New York, NY 10012 \\
                USA }
\email{tschinkel@cims.nyu.edu}
\address{
Simons Foundation \\
160 Fifth Avenue \\
New York, NY 10010 \\
USA}

\title{Embedding pointed curves in K3 surfaces}

\begin{document} 
\date{\today}

\maketitle

\section{Introduction}

Let $D$ be a smooth projective curve of genus $g$ over an algebraically closed
field of characteristic zero.
Let $\cM_g$ denote the moduli stack of such curves.
Let $(S,h)$ be a polarized K3 surface of genus $g$, i.e., $h$ is ample
and primitive with $h^2=2g-2$.  Let $\cF_g$
denote the moduli stack of such surfaces.
Now suppose $D \subset S$ with $[D]=h$; let $\cP_g$ denote the moduli space of
such pairs $(S,D)$, $\varphi_g:\cP_g \ra \cM_g$ the forgetful map, and
$\cK_g \subset \cM_g$ its image.
These morphisms have been studied systematically by Mukai; we review
this in more detail in Section~\ref{S:GB}.  One of the highlights of this theory
is the {\em birationality} of $\varphi_g$ when $g=11$.  

We propose variations
on this construction.  Let $\cF_{\Lambda}$ denote a moduli space of lattice polarized 
K3 surfaces, where $\Lambda\supset \left<h,R\right>$ with $h$ a polarization of degree $2g-2$
and $R$ an indecomposable $(-2)$-class (smooth rational curve) with $n:=h\cdot R>0$.  Let $\cP_{\Lambda}$
denote the space of pairs $(S,D)$, where $S\in \cF_{\Lambda}$ and $D\in |h|$ is
smooth and meets $R$ transversally.  
Thus we have 
$$\begin{array}{rcl}
\varphi_{\Lambda}: \cP_{\Lambda} & \ra & \cM_{g;n}:=\cM_{g,n}/\mathfrak{S}_n \\
          (S,D) & \mapsto & (D;D\cap R)
\end{array}
$$
as well as the morphism keeping track of the pointed rational curve
$$\begin{array}{rcl}
\widehat{\varphi_{\Lambda}}: \cP_{\Lambda} & \ra & (\cM_{g,n}\times \cM_{0,n})/\mathfrak{S}_n \\
          (S,D) & \mapsto & ((D;D\cap R),(R;D\cap R))
\end{array}
$$
Suppose that $\Lambda=\left<h,R\right>$ and the source and target moduli
spaces have the same dimension; thus $2g+n=21$ in the first case
and $g+n=12$ in the second.  What are $\deg(\varphi_{\Lambda})$
and $\deg(\widehat{\varphi_{\Lambda}})$?
We answer these questions for $g=7$; see Remark~\ref{rema:71} 
and Section~\ref{sect:related}.

Work of Green and Lazarsfeld \cite{GL} shows how lattice polarizations control
the Brill-Noether properties of curves on K3 surfaces.  For example, a smooth
curve $D \subset S$ admits a $g^1_d$ with $2d\le g+1$ only if there is an elliptic
fibration $S\ra \bP^1$ with $D$ as a multisection of degree $d$.  
Hence for suitable lattice polarizations $\varphi_{\Lambda}$ maps to Brill-Noether
strata of $\cM_g$.  
Our variation is relevant to understanding the specialization
of Mukai's theory over these strata.

In this note, we focus on specific lattices with $g=7$ and $11$;
see Theorems~\ref{theo:main7} and \ref{theo:specialize}.  
The approach is to specialize Mukai's construction in genus $11$
to the pentagonal locus, degenerate this to a carefully chosen
stable curve of genus $11$, and then deform a related stable curve
of genus $7$ together with the ambient K3 surface.

Our motivation for considering this geometry comes from arithmetic.  
The constructions in this paper have strong implications for the 
structure of spaces of sections for del Pezzo surface fibrations 
over $\bP^1$.  Details appear in \cite{HT12}.  

\subsection*{Acknowledgments}  Andrew Kresch provided invaluable assistance on this
project and in particular, the computations in the enclosed appendix. 
We benefited from conversations with Shigeru Mukai, Frank-Olaf Schreyer, and
Alessandro Verra.  
The first author is supported by NSF grants
0968349, 0901645, and 1148609; the second author is supported by NSF grants 0739380, 0968349, and 1160859.

\section{Geometric background}
\label{S:GB}
Consider the moduli space 
$M_{2,h,s}(S)$ of rank-two simple sheaves $E$ on a K3 surface $S$ with $c_1(E)=h$ and 
$\chi(E)=2+s$.  This is holomorphic symplectic of dimension $h^2-4s+2=2g-4s$,
provided this expression is non-negative.  When $g-1=2s$ the moduli
space is again a K3 surface, isogenous to $S$ \cite{Mukai87}.  

Let $M_{2,\omega}(D)$ denote the moduli space of semistable rank-two 
vector bundles on $D$ with canonical determinant.  Consider the non-abelian
Brill-Noether loci defined by Mukai \cite{Mukai01}
$$M_{2,\omega,s}(D):=\{E \in M_{2,\omega}(D): h^0(E) \ge 2+s \},$$
which has expected dimension $3g-3-\binom{s+3}{2}.$

Restricting bundles from $S$ to $D$
yields examples of non-abelian Brill-Noether loci with dimensions
frequently larger than expected.  Mukai has developed a program, supported
by beautiful examples, that seeks to characterize $\cK_g\subset \cM_g$ 
in terms of special non-abelian Brill-Noether loci.

A particularly striking result along these lines is:
\begin{theo} \cite[Thm.~1]{Mukai96}
Let $D$ be a generic curve of genus eleven.  Then there exist a genus eleven K3 surface $(S,h)$ and
an embedding $D\hookrightarrow S$, which are unique up to isomorphisms.  Furthermore, we can 
characterize $S$ as $M_{2,h,5}(T)$, where $T=M_{2,\omega,5}(D)$, which is also a genus eleven
K3 surface.  
\end{theo}
Thus $\varphi_{11}: \cP_{11} \ra \cM_{11}$ is birational and the K3 surface can be recovered
via moduli spaces of vector bundles on $D$.  
The theorem remains true for the generic hexagonal curves of genus eleven; see \cite[Thm.~3]{Mukai96}.
Hexagonal curves form a divisor in $\cM_{11}$, so $\varphi_{11}$ remains an isomorphism
over the generic point of the hexagonal locus.  We will analyze what happens over the
pentagonal locus. 

We shall need a similar result along these lines in genus seven \cite{MukaiAJM}.  
Let $\OGr(5,10) \subset \bP^{15}$ denote the orthogonal Grassmannian, 
parametrizing five dimensional isotropic subspaces for a ten dimensional
non-degenerate quadratic form.  
Let $G$ denote the corresponding orthogonal group.
The intersection of $\OGr(5,10)$ with a
generic six dimensional subspace $\Pi_6 \subset \bP^{15}$ is a canonical
curve of genus seven; its intersection with a seven dimensional subspace
$\Pi_7 \subset \bP^{15}$ is a K3 surface of genus seven.  
Thus we obtain rational maps
$$\begin{array}{rcl}
\Gr(7,16)/G & \dashrightarrow &\cM_7 \\
\Gr(8,16)/G & \dashrightarrow & \cF_7 \\
\Fl(7,8,16)/G & \dashrightarrow & \cP_7 
\end{array}$$
where $\Fl(7,8,16) \subset \Gr(7,16) \times \Gr(8,16)$ is the flag variety.  
Mukai shows that each of these is birational.

\section{Degeneration of Mukai's construction over the pentagonal locus}
\label{S:degenerate}

Let $V\subset \bP^6$ denote the Fano threefold of index two obtained as a generic linear section
of the Grassmannian $\Gr(2,5)$.  
\begin{prop} \label{prop:toGr}
Let $D$ be a pentagonal curve of genus eleven, i.e., $D$ admits a basepoint free $g^1_5$.
Assume $D$ is generic.  
Consider
$$D \hookrightarrow \bP^6,$$
the embedding induced by the linear series adjoint to the $g^1_5$.
This admits a canonical factorization
$$D \subset V \subset \bP^6.$$
\end{prop}
\begin{proof}
Let $\phi_1:D \ra \bP^1$ be the degree five morphism.
Write $(\phi_1)_*\cO_D=\cO_{\bP^1} \oplus \cF$ using the trace homomorphism;
by relative duality, we have
$$(\phi_1)_*\omega_D= \omega_{\bP^1} \oplus \cF^{\vee}\otimes \omega_{\bP^1}$$
and
$$\Gamma(D,\omega_D)=\Gamma(\bP^1, \omega_{\bP^1} \oplus \cF^{\vee} \otimes \omega_{\bP^1}).$$
Moreover $\cF^{\vee} \otimes \omega_{\bP^1}$ is globally generated, as the divisors in the
basepoint free $g^1_5$ impose four conditions on the canonical series.
Hence the canonical embedding factors
$$D \hookrightarrow \bP(\cF \otimes T_{\bP^1}) \rightarrow \bP^{10}.$$

Our genericity assumption is that the vector bundles above are as `balanced' as possible, i.e.,
$$\cE:=\cF \otimes T_{\bP^1}\simeq \cO_{\bP^1}(-1) \oplus \cO_{\bP^1}(-2)^{\oplus 3}.$$
Since $\cE^{\vee} \otimes \cO_{\bP^1}(-1)$ is still globally generated, 
we also have a factorization of the adjoint morphism
$$D \hookrightarrow \bP(\cE \otimes \cO_{\bP^1}(1)) \ra \bP^6;$$
the image of the projective bundle is a cone over the Segre threefold $\bP^1 \times \bP^2$.  

The tautological exact sequence on $p:\bP(\cE) \ra \bP^1$ takes the form
$$0 \ra \cO_{\bP(\cE)}(-1) \ra p^*\cE \ra Q \ra 0;$$
twisting yields
$$0 \ra \cO_{\bP(\cE)}(-1) \otimes p^*\cO_{\bP^1}(2) \ra p^*(\cE \otimes \cO_{\bP^1}(2)) \ra Q':=Q\otimes p^*\cO_{\bP^1}(2) \ra 0.$$
Note that $\rank(Q')=3$, $h^0(Q')=5$, $Q'$ is globally generated, and
$c_1(Q')=c_1(\cO_{\bP(\cE)}(1))-p^*c_1(\cO_{\bP^1}(1)).$
Restricting to $D\subset \bP(\cE)$ gives 
$$c_1(Q'|D)=K_D-g^1_5,$$
the adjoint divisor.  The classifying map for $Q'|D$ gives a morphism
$$D \ra \Gr(2,5)$$
factoring through a codimension three linear section, which is $V$.
\end{proof}

\begin{rema} \label{rema:noMaroni}
We isolate where the generality assumption is used:  It is
necessary that $\cE$ not admit any summands of degree $\leq -3$, or equivalently,
$$(\phi_1)_*\omega_D=\omega_{\bP^1} \oplus \cO_{\bP^1}(1) \oplus \cO_{\bP^1}(2)^{\oplus 3}.$$
\end{rema}

\begin{theo} \label{theo:pentagonal}
Let $D$ be a generic pentagonal curve of genus eleven and
$(\phi_1,\phi_2):D \ra \bP^1 \times V$ the embedding given by
the degree five covering and Proposition~\ref{prop:toGr}.
We have
\begin{itemize}
\item{$\varphi_{11}^{-1}(D) \simeq \bP^2$, specifically,
the K3 surfaces containing $D$ are the codimension-two linear sections of 
$\bP^1 \times V$ containing $D$.}
\item{Fix two distinct points $d_1,d_2 \in D$ with $\phi_1(d_1)=\phi_1(d_2)=:p$;
given a conic $Z$ satisfying 
$$d_1,d_2 \in Z \subset \phi_1^{-1}(p),$$
there exists a unique K3 surface $S$ containing $D$ and $Z$.}
\end{itemize}
\end{theo}
The K3 surface $S$ has lattice polarization
\begin{equation} \label{eqn:LP}
\begin{array}{r|ccc}
	& h  & E & Z \\
\hline
h  & 20 & 5 & 2 \\
E  & 5 &  0 & 0 \\
Z  & 2 &  0 & -2
\end{array}
\end{equation}
where $E$ is the elliptic fibration inducing the $g^1_5$ on $D \in |h|$. 
\begin{proof}
We present the basic geometric set-up.  We have
$$\bP^1 \times V \subset \bP^1 \times \bP^6 \subset \bP^{13}$$
where the last inclusion is the Segre embedding.  
Given a flag
$$\bP^{10} \subset \bP^{11} \subset \bP^{13},$$
intersecting with $\bP^1 \times V$ yields
$$D \subset S \subset \bP^1 \times V,$$
where $D$ is a canonically embedded pentagonal curve of genus five, and $S$ is a K3
surface with lattice polarization: 
$$\begin{array}{r|cc}
	& h  & E \\
\hline
h     & 20 & 5 \\
E     & 5  & 0 
\end{array}
$$
Fixing $D$, the K3 surfaces $S$ containing $D$ correspond to the $\bP^{11}$'s
in $\bP^{13}$ containing the fixed $\bP^{10}$, which are parametrized by $\bP^2$.  
This proves the first assertion.

We prove the existence assertion of the second part. The assumption $\phi_1(d_1)=\phi_1(d_2)=:p$ means that
$$d_1,d_2 \in \{p \} \times V \subset \bP^1 \times V.$$
Two-dimensional linear sections $S$ containing $D$ and the desired conic correspond to conics
$Z\subset V$ passing through $d_1,d_2 \in V$.  Indeed, $S$ may be recovered 
from $Z$:
$$S=\mathrm{span}(D\cup_{d_1,d_2}Z) \cap (\bP^1 \times V).$$

Note that each such $S$ is automatically regular along $D$.  
For generic pentagonal $D$ and $p \in \bP^1$, $S$ must be regular everywhere.
We see this by a parameter count.  Pairs $(S,D)$ where $S$ is a K3 surface
with lattice polarization (\ref{eqn:LP}) and $D\in |h|$ depend on $17+11=28$
parameters.  Pairs $(D,p)$ where $D$ admits a degree-five morphism to $\bP^1$
and $p \in \bP^1$ depend on $30-3+1=28$ parameters.  Thus the $(D,p)$ 
arising from singular K3 surfaces cannot be generic.  

The uniqueness assertion boils down 
to an enumerative problem:  How many conics pass through two
prescribed generic points of $V$?  
Recall that the Fano threefold $V$ may be obtained from a smooth quadric hypersurface $Q\subset \bP^4$
explicitly (see e.g. \cite[p.~173]{Isk}):
$$\begin{array}{ccccc}
	&	& \Bl_{\ell}(V)=\Bl_{m}(Q) &  & \\
        & \swarrow & 			& \searrow & \\
V       &          &                    &          & Q
\end{array}
$$
The rational map $V\dashrightarrow Q$ arises from projecting from a line $\ell \subset V$;
the inverse $Q \dashrightarrow V$ is induced by the linear series of quadrics vanishing
along a twisted cubic curve $m \subset Q$.  Note that $V$ contains a two-parameter family of
lines which sweep out the threefold, so we may take $\ell$ disjoint from $d_1$ and $d_2$.

Using this modification, our enumerative problem may be transferred to $Q$:  How many
conics in $Q$ pass through two prescribed generic points $d_1,d_2 \in Q$ and meet a twisted cubic 
$m \subset Q$ twice at unprescribed points?  There is one such curve.  Indeed,
projection from the line spanned by the two prescribed points gives a morphism
$$m \ra \bP^2$$
with image a nodal cubic.  Let $m_1,m_2 \in m$ be the points lying over the node;
the set $\{m_1,m_2,d_1,d_2 \} \subset Q \subset \bP^4$ lies in a plane $P$.  The stipulated conic
arises as the intersection
$P \cap Q.$
\end{proof}

\section{Specializing pentagonal curves of genus $11$}
\label{S:moredegenerate}

For our ultimate application to genus $7$ curves, we must specialize 
the construction of Section~\ref{S:degenerate} further:
\begin{theo} \label{theo:specialize}
Let $C_1$ denote a generic curve of genus five, $c_1,\ldots,c_7 \in C_1$
generic points, and $\psi:C_1 \ra \bP^1=:R'$ a degree four morphism.
Let $D=C_1 \cup_{c_j=\psi(c_j)} R'$ denote the genus eleven nodal curve
obtained by gluing $C_1$ and $R'$, which is automatically pentagonal. 
Fix an additional generic point $c_0 \in C_1$, and
write $\psi^{-1}(\psi(c_0))=\{c_0,c'_1,c'_2,c'_3 \}$.  
Then there exists a unique embedding $D \hookrightarrow S$ where $S$ is a K3 surface
with lattice polarization 
$$\begin{array}{r|cccc}
 & f & C_1 & C_2 & R' \\
\hline
f & 4 & 7 & 1 & 6 \\
C_1 & 7 & 8 & 3 & 7 \\
C_2 & 1 & 3 & -2 & 0 \\
R' & 6 & 7 & 0 & -2
\end{array}
$$
where $C_2\simeq \bP^1$ intersects $C_1$ at $\{c'_1,c'_2,c'_3 \}$.
\end{theo}
\begin{proof}
We claim that the argument of Theorem~\ref{theo:pentagonal} applies, as $D$ satisfies
Remark~\ref{rema:noMaroni}.  
Indeed, it suffices to show that
$$h^0(\omega_D(-2g^1_5))=3;$$
if $(\phi_1)_*\omega_D$ failed to have the expected decomposition then 
$$h^0(\bP^1,(\phi_1)_*\omega_D)>3,$$
a contradiction.

If $h^0(\omega_D(-2g^1_5))>3$ then 
$$m:=h^0(\cO_D(2g^1_5))=h^1(\cO_D(2g^1_5))>3;$$
clearly $2g^1_5$ is basepoint free, so we have a morphism
$$D \ra \bP^{m-1}, \quad m\ge 4.$$
The image of $R'$ under this morphism is a plane conic,
hence the images of $c_1,\ldots,c_7$ are distinct coplanar
points.  This means that on $C_1$
$$2g^1_4-c_1-\cdots-c_7, \quad g^1_4=g^1_5|C_1$$
is effective, contradicting the genericity of $c_1,\ldots,c_7$. 

Now we take $d_1=c_0$ and $d_2=\psi(c_0) \in R' \simeq \bP^1$.
Thus $D$ is contained in a distinguished surface $S$ containing a
rational curve $Z \ni d_1,d_2$.  
Hence $S$ has lattice polarization
$$
\begin{array}{r|cccc}
	& C_1 & R'  & E & Z \\
\hline
C_1  & 8 & 7 & 4& 1 \\
R' &  7  & -2 & 1  & 1   \\
E  & 4 & 1 &  0 & 0 \\
Z  & 1 & 1 & 0 & -2
\end{array}
$$
containing the intersection matrix (\ref{eqn:LP}).
Using the identifications
$$h=C_1 + R', \ E=C_1+C_2-f,\ Z=C_1-f,\ R'=R',$$
we obtain the desired lattice polarization.

It remains to show that for generic inputs the surface $S$ is in fact
smooth.  As before, this follows from a parameter count.  The data
$$(C_1,c_0,c_1,\ldots,c_7)$$
consists of a genus five curve ($12$ parameters), a choice of $g^1_4$ on
that curve ($1$ parameter), and eight generic points, for a total
of $21$ parameters.  On the output side, we have a K3 surface with the prescribed
lattice polarization of rank four
($16$ parameters) and a curve in the
linear series $|C_1|$ ($5$ parameters).  Thus for generic input data, the resulting
surface is necessarily smooth. 
\end{proof}

\section{Genus $7$ K3 surfaces and rational normal septic curves}
Consider the moduli space of lattice-polarized K3 surfaces of type
$$\Lambda':=\begin{array}{r|cc}
	& C & R' \\
\hline
C   & 12 & 7 \\
R'  & 7 & -2
\end{array}
$$
and let $\cP_{\Lambda}$ denote the moduli space of pairs $(S,D)$, where $S$ is such a 
K3 surface and $D \in |C|$ is smooth and meets $R'$ transversely.
\begin{prop} \label{prop:GF}
The forgetting morphism
$$
\varphi_{\Lambda'}:\cP_{\Lambda'}  \ra  \cM_{7;7}
$$
is generically finite.
\end{prop}
Note that the varieties are both of dimension $25$.  
The main ingredient of the proof is:
\begin{lemm} \label{lemm:7pts}
Let
$\cM_{0,7}(\OGr(5,10),7)$
be the moduli space of pointed mappings of degree $7\ell$,
where $\ell \in H_2(\OGr(5,10),\bZ)$ is Poincar\'e
dual to the hyperplane class $h$.  Then the evaluation map
$$\ev^7:\cM_{0,7}(\OGr(5,10)),7) \ra \OGr(5,10)^7$$
is dominant.  
\end{lemm}
\begin{proof}
Recall that $\dim(\OGr(5,10))=10$ and the canonical class is
$$K_{\OGr(5,10)}=-8h.$$
The expected dimension of the moduli space is $70$, so we expect
$\ev^7$ to be generically finite.

By \cite[15.7]{dJHS}, the evaluation map
$$\ev:\cM_{0,7}(\OGr(5,10)),1) \ra \OGr(5,10)$$
is surjective, i.e., pointed lines dominate $\OGr(5,10)$.  It follows 
from generic smoothness that $H^1(N_{\ell/\OGr(5,10)}(-1))=0$ 
\cite[p.33]{AK} and 
$$N_{\ell/\OGr(5,10)} \simeq \cO_{\bP^1}^{\oplus 3} \oplus \cO_{\bP^1}(1)^{\oplus 6}.$$

Fix a generic chain of seven lines in $\OGr(5,10)$
$$\cC_0:=\ell_1 \cup_{e_{12}} \ell_2 \cup_{e_{23}} \ldots \cup_{e_{67}}\ell_7,$$
as well as generic points $r_j(0) \in \ell_j$.  Consider the following moduli
problem:  Fix a smooth base scheme $B$ with basepoint $0$ and morphisms
$r_j:B \ra \OGr(5,10), j=1,\ldots,7$ mapping $0$ to the point $r_j(0)$
specified above.  We are interested in subschemes 
$$\begin{array}{ccccc}
r_1,\ldots,r_7 & \subset & \cC & \subset & \OGr(5,10) \times B \\
		& \searrow &\downarrow   & \swarrow   &   \\
                &          & B  &     & 
\end{array}
$$
all flat over $B$, with the distinguished fiber of $\cC$ equal to $\cC_0$.  
The deformation theory of this Hilbert scheme as a scheme over $B$ \cite[\S 6]{AK} is governed by
the tangent space 
$$\Gamma(N_{\cC_0/\OGr(5,10)}(-r_1(0)-\cdots-r_7(0)))$$
and the obstruction space
$$H^1(N_{\cC_0/\OGr(5,10)}(-r_1(0)-\cdots-r_7(0))).$$
If the latter group is $0$, the Hilbert scheme is flat over $B$ of the expected dimension
$$\chi(N_{\cC_0/\OGr(5,10)}(-r_1(0)-\cdots-r_7(0)))=0.$$
Since $\cC_0$ is a chain of rational curves, it suffices to exclude higher cohomology
on each irreducible component $\ell_i$.  For $i=1,7$ we have
$$N_{\cC_0/\OGr(5,10)}|\ell_i=\cO_{\bP^1}^3 \oplus \cO_{\bP^1}(1)^5 \oplus \cO_{\bP^1}(2)$$
which has trivial higher cohomology, even after twisting by $\cO_{\ell_i}(-r_i(0))$.
For $i=2,\ldots,6$, we have
$$N_{\cC_0/\OGr(5,10)}|\ell_i=\cO_{\bP^1}^3 \oplus \cO_{\bP^1}(1)^4 \oplus \cO_{\bP^1}(2)^2,$$
reflecting the two attaching points on $\ell_i$.  Twisting by $\cO_{\ell_i}(-r_i(0))$,
we have no higher cohomology as well.  

Choosing $r_j:B \ra \OGr(5,10)$ suitably generic, we conclude there is a zero dimensional
collection of rational septic curves passing through seven generic points of $\OGr(5,10)$.
\end{proof}

We complete the proof of Proposition~\ref{prop:GF}.
Let $D$ denote a generic curve of genus seven and $r_1,\ldots,r_7 \in D$ generic
points.  As we recalled in Section~\ref{S:GB}, $C'$ arises as a linear section of $\OGr(5,10)$,
and Lemma~\ref{lemm:7pts} yields a septic rational curve $R' \subset \OGr(5,10)$ 
containing these seven points.  The intersection
$$\OGr(5,10) \cap \mathrm{span}(D \cup R')$$
is a K3 surface, with the prescribed lattice polarization.

\begin{rema}  \label{rema:71}
What is the degree of $\varphi_{\Lambda'}$?  The birationality results quoted in Section~\ref{S:GB}
imply that the degree equals
the degree of the generically-finite mapping
$$\ev^7:\cM_{0,7}(\OGr(5,10)),7) \ra \OGr(5,10)^7.$$
See the Appendix for a proof that $\deg(\ev^7)=71$.  
\end{rema}

\section{From genus $11$ to genus $7$}

\begin{theo} \label{theo:main7}
Let $\pi:C\ra R'=\bP^1$ be a tetragonal curve of genus seven, $C\subset \bP^3$
the adjoint embedding as a curve of degree eight, $c_1,\ldots,c_7 \in C$
generic points.  Then there exists a unique embedding 
$$\varpi:R' \hookrightarrow \bP^3, \quad \varpi(\pi(c_j))=c_j$$
such that there exists a quartic surface $S$ containing both $C$ and $R'$.  
\end{theo}
The K3 surfaces in this case have lattice polarization:
\begin{equation} \label{eqn:r3}
\Lambda=\begin{array}{r|ccc}
   & f & C & R' \\
\hline
f & 4 & 8 & 6 \\
C  & 8 & 12 & 7 \\
R' & 6 & 7 & -2
\end{array}
\end{equation}
\begin{proof}
We regard Theorem~\ref{theo:specialize} as a special case of this, via
the specialization
$$C \leadsto C_1 \cup C_2.$$
We can deform the K3 surfaces in Theorem~\ref{theo:specialize}
to the K3 surfaces with lattice polarization (\ref{eqn:r3}).  
Under this deformation, $R'$ deforms to a rational curve in the nearby
fibers; smooth rational curves in K3 surfaces always deform provided
their divisor classes remain algebraic.  

We claim that $C_1 \cup C_2$ deforms to a generic 
tetragonal curve of genus seven, with seven generic marked points
traced out by $R'$.

For our purposes, we would like to restrict the curve $C$ to be tetragonal.  This
is equivalent (see \cite{GL}) to imposing a lattice polarization of the type
$$\begin{array}{r|ccc}
	& C & R' & E \\
\hline
C   & 12 & 7 & 4 \\
R'  & 7 & -2 & a \\
E  & 4  & a & 0
\end{array}
$$
where $E$ is the class of a fiber of an elliptic fibration.  Restricting $\varphi_{\Lambda'}$
to each such lattice polarization, we obtain a generically finite morphism to
the Brill-Noether divisor $\cT \subset \cM_{7;7}$ corresponding to the
tetragonal curves.  In our geometric analysis, it will be convenient to use
a different basis
$$\begin{array}{r|ccc}
	& C & R' & f \\
\hline
C   & 12 & 7 & 8 \\
R'  & 7 & -2 & 7-a \\
f  & 8  & 7-a & 4 
\end{array}
$$
where $f=C-E$.  We will restrict our attention to the particular component 
with $a=1$, i.e., $R'$ is a {\em section} of the elliptic fibration inducing 
the $g^1_4$ on $C$. 
This is $\cP_{\Lambda}$, where $\Lambda$ is defined in (\ref{eqn:r3}).

Theorem~\ref{theo:main7} asserts that the morphism
$$\varphi_{\Lambda}:\cP_{\Lambda} \ra \cT$$
has degree one.  Consider the specialization 
$C \leadsto C_1 \cup C_2$ as above.  After specialization, Theorem~\ref{theo:specialize}
guarantees a {\em unique} K3 surface containing $C_1 \cup C_2$.  Thus $\varphi_{\Lambda}$
has degree at least one.  If the degree were greater than one, then a generic
point of $\cT$ would yield at least {\em two} surfaces $S$ and $S'$,
with specializations $S_0$ and $S'_0$.  These are necessarily K3 surfaces,
by the following result about degenerate quartic surfaces:

\begin{lemm}
Consider the projective space $\bP^{34}$ parametrizing all quartic surfaces
in $\bP^3$.  Let $\Sigma \subset \bP^{34}$ denote the surfaces with
singularities worse than ADE singularities.  Then the codimension of $\Sigma$
is at least four.
\end{lemm}
\begin{proof}
We first address the non-isolated case.
The reducible surfaces of this type---unions
of two quadric surfaces or a plane and a cubic---have codimension much larger
than four.  
The irreducible surfaces are classified by Urabe \cite[\S 2]{Urabe},
building on the work of numerous predecessors over the last century.  In each case, the
codimension is at least four.  

There is also a substantial literature on the classification of 
isolated singularities of quartic surfaces,
e.g., \cite{JShah,Deg,Ishii}.  However, it will be more convenient for
us to give a direct argument, rather than refer to the details of a 
complete classification.

Since $\bP^3$ is homogeneous, it suffices to show that the surfaces
with a non ADE singularity at $p=[0,0,0,1]$ have codimension $\ge 3$ in the 
locus of surfaces singular at $p$.  The equations with multiplicity $>2$ 
at $p$ have codimension $\ge 6$ and thus can be ignored; the equations of multiplicity
two having non-reduced tangent cone at $p$ have codimension $3$.  The equations
with an isolated singularity of multiplicity two at $p$ (and reduced tangent cone)
are of ADE type. 
\end{proof}

It only remains to exclude ramification at the generic point of the stratum,
i.e., that the fibers of $\varphi_{\Lambda}$ have zero dimensional tangent space.
Section~\ref{S:defcomp} is devoted to proving this.  
\end{proof}

We see these results as special cases of a more general statement
of Mukai \cite[\S 10]{Mukai96}:
\begin{quote}
$\varphi_{13}:\cP_{13} \ra \cK_{13}$ is birational onto its image.
\end{quote}
This might be approached via a degeneration argument as follows.
Consider the stable curve $B=C_1 \cup C_2 \cup R'$ as above, i.e.,
\begin{itemize}
\item{$\psi:C_1\ra \bP^1$ is tetragonal of genus five;}
\item{$R'\simeq \bP^1$ and is glued to $C_1$ at seven generic
points via $\psi$;}
\item{$C_2 \simeq \bP^1$, is disjoint from $R'$, and meets $C_1$
in three points on a generic fiber of $\psi$.}
\end{itemize}
Note that $B$ has arithmetic genus thirteen and is contained in the
image of $\varphi_{13}$, or more precisely, a suitable extension of $\varphi_{13}$
which we now describe.

Let $\ocP_g$ denote the irreducible component moduli stack of stable log pairs $(S,C)$
containing $\cP_g$.  There is an extension
$$\varphi_g:\ocP_g \ra \ocM_g$$
to the moduli space of stable curves.  Suppose there is a point $C_{\eta} \in \cM_g$
such that $C_{\eta} \subset S_{\eta},S'_{\eta}$, both K3 surfaces of genus $g$.
Suppose there is a specialization $C_{\eta} \leadsto C_0$ 
admitting a specialization $S_{\eta} \leadsto S_0$ to a K3 surface containing $C_0$.
Then there exists a specialization $S'_{\eta} \leadsto S'_0$ to a K3 surface 
containing $C_0$.  In other words, the part of $\ocP_g$ arising from K3 surfaces
has no `holes'.  The reason to expect this is that the K3 surfaces $S_{\eta}$ and
$S'_{\eta}$ ought to be isogenous, and this isogeny should also specialize, so
one cannot become singular without the other becoming singular.  

This would require additional argument, e.g., by identifying a situation
where we can analyze {\em a priori} the fibers of $\varphi_g$.
This might be possible with a suitable GIT construction of $\cP_{13}$.

\section{Deformation computations}
\label{S:defcomp}
\subsection{Generalities}
We review the formalism of deformations of pairs, following
\cite{Kawamata}.  For the moment, $S$ denotes a smooth projective 
variety and $D \subset S$ a reduced normal crossings divisor.  
We have exact sequences
$$0 \ra T_S(-D) \ra T_S\left<-D\right> \ra T_D \ra 0$$
and 
$$0 \ra T_S\left<-D\right> \ra T_S \ra N_{D/S} \ra 0,$$
where $T_S\left<-D\right>$ means vector fields on $S$ with
logarithmic zeros along $D$.  
The tangent space to the deformation space of $(S,D)$ is given by
$H^1(T_S\left<-D\right>)$.  The sequence
$$\Gamma(N_{D/S}) \ra H^1(T_S\left<-D\right>) \ra H^1(T_S)$$
may be interpreted `first-order deformations of $(S,D)$ leaving
$S$ unchanged arise from deformations of $D\subset S$'; the sequence
$$H^1(T_S(-D)) \ra H^1(T_S\left<-D\right>) \ra H^1(T_D)$$
means that we may interpret $H^1(T_S(-D))$ as `first-order deformations
of $(S,D)$ leaving $D$ unchanged'.  

Our situation is a bit more complicated as our boundary consists of
{\em two} log divisors deforming independently.  If $D=\{x=0\}$
and $R'=\{y=0\}$ meet transversally at $x=y=0$ then 
$T_S\left<-D\right>\left<-R'\right>$
is freely generated by $x\frac{\partial}{\partial x}$ and
$y\frac{\partial}{\partial y}.$
First-order deformations of $(S,D,R')$ are given by
$H^1(T_S\left<-D\right>\left<-R'\right>)$.

We analyze the ramification of the forgetting morphism $\varphi$ from the deformation space
of $(S,D,R)$ to the deformation space of $(D, D \cap R')$.
A slight variation on one of the standard exact sequences above
$$0 \ra T_S\left<-R\right>(-D) \ra T_S\left<-D\right>\left<-R'\right> \ra T_D\left<-R'\right> \ra 0$$
induces the differential
$$d\varphi:H^1(T_S\left<-D\right>\left<-R'\right>) \ra H^1(T_D\left<-D\cap R\right>).$$
Hence the kernel is given by $H^1(T_S\left<-R'\right>(-D))$, which is zero precisely
when $\phi$ is unramified.
The exact sequence
$$0 \ra T_S\left<-R'\right>(-D) \ra T_S(-D) \ra N_{R'/S}(-D\cap R') 
\ra 0$$
induces
$$
\begin{array}{r}
\Gamma(N_{R'/S}(-D\cap R')) \ra \\
 H^1(T_S\left<-R'\right>(-D)) 
\ra H^1(T_S(-D)) \stackrel{h}{\ra} H^1(N_{R'/S}(-D \cap R')).
\end{array}
$$
Thus $\varphi$ is unramified if $\Gamma(N_{R'/S}(-D\cap R'))=0$ and
$h$ is injective.  

\subsection{Our specific situation}
\label{SS:OSS}
We specify what the various objects are: First,
$$S=\{F_1=F_2=0 \} \subset \bP^1 \times \bP^3.$$
is a smooth complete intersection of forms of bidegree $(1,2)$.
The divisor 
$$D=\{L=0 \} \subset S$$
is a hyperplane section, i.e., $L$ is of bidegree $(1,1)$.  
The divisor $R'\subset S$ is a smooth rational curve of bidegree $(1,6)$,
meeting $D$ transversally, in seven points.  

Since $\deg(N_{R'/S}(-D \cap R'))=-2-7$ it has no global sections;
the higher cohomology is computed via Serre duality
$$H^1(N_{R'/S}(-D\cap R')=
\Gamma(\cO_{R'}(D\cap R'))^{\vee}=\Gamma(\cO_{\bP^1}(7))^{\vee}.$$
We claim that 
$H^1(T_S(-D))$ is eight dimensional.
\begin{rema}
$H^1(T_S(-D))$ is therefore eight dimensional when $S$
is a generic K3 surface of degree $12$.  This has a nice global
interpretation via the Mukai construction \cite{Mukai96}:  
$D$ (resp.~$S$) is a codimension nine (resp.~eight) linear section
of the orthogonal Grassmannian $\OGr(5,10)$, so the 
K3 surfaces containing a fixed $D$ are parametrized by $\bP^8$.
\end{rema} 
To prove the claim, use the normal bundle exact sequence
$$0 \ra T_S(-D) \ra T_{\bP^1 \times \bP^3}(-1,-1)|S \ra
 N_{S/\bP^1 \times \bP^3}(-1,-1) \ra 0,$$
the Koszul complex for $\{F_1,F_2\}$
$$0 \ra \cO_{\bP^1 \times \bP^3}(-2,-4) \ra 
\cO_{\bP^1 \times \bP^3}(-1,-2)^{\oplus 2} \ra
\cO_{\bP^1 \times \bP^3} \ra \cO_S \ra 0$$
and its twist by $T_{\bP^1 \times \bP^3}(-1,-1).$
Applying the Kunneth formula to compute the cohomologies of
the twists of $T_{\bP^1 \times \bP^3}$, we find that 
$$\Gamma(T_{\bP^1 \times \bP^3}(-1,-1)|S)=
H^1(T_{\bP^1 \times \bP^3}(-1,-1)|S)=0.$$
We also find that
$$\Gamma(N_{S/\bP^1 \times \bP^3}(-1,-1))=\Gamma(\cO_S(0,1))^{\oplus 2}
=\Gamma(\cO_{\bP^3}(1))^{\oplus 2},$$
which is isomorphic to $H^1(T_S(-D))$ by the vanishing above.

In concrete terms, 
the infinitesimal deformations corresponding
to $H^1(T_S(-D))$ take the form
\begin{equation} \label{eqn:Seps}
S(\epsilon)=\{F_1+\epsilon L G_1=F_2+\epsilon L G_2=0\}, \quad 
G_1,G_2 \in \Gamma(\cO_{\bP^3}(1)).
\end{equation}

The exact sequence above therefore takes the form
$$0  \ra H^1(T_S\left<-R'\right>(-D)) 
\ra \Gamma(\cO_{\bP^3}(1))^{\oplus 2} \stackrel{h}{\ra}
\Gamma(\cO_{\bP^1}(7)),
$$
where $h$ captures the obstructions to deforming
the rational curve along an infinitesimal deformation of $S$.
Precisely, suppose we start off with
\begin{itemize}
\item{$(\bP^1,r_1,\ldots,r_7)$ a pointed rational curve;}
\item{a K3 surface $S=\{F_1=F_2=0 \} \subset \bP^1 \times \bP^4$;}
\item{a nodal hyperplane section curve $D\subset S$ with
equation $L=0$;}
\item{a morphism $\iota:\bP^1 \ra S$ of bidegree $(1,6)$ with
$d_j=\iota(r_j) \in D$.}
\end{itemize}
Fix $\epsilon$ to be a small parameter, or in algebraic terms,
a nilpotent with $\epsilon^2=0$.  
We analyze first-order deformations of
\begin{itemize}
\item{the K3 surface $S(\epsilon)$ as in (\ref{eqn:Seps});}
\item{the morphism $\iota(\epsilon):\bP^1 \ra S_{\epsilon}$
satisfying the constraint
\begin{equation} \label{eqn:constraint}
d_j=\iota(\epsilon)(r_j).
\end{equation}}
\end{itemize}

Let $[t,u]$ be coordinates on $\bP^1$ and write the seven distinct points
$[r_j,1],j=1,\ldots,7$.  Let $P_1,\ldots,P_7 \in k[t,u]_6$ be
homogeneous polynomials of degree six such that $P_i(r_j,1)=\delta_{ij}$;
these form a basis of $k[t,u]_6$.   We factor $\iota$ as the composition of the $6$-uple
embedding
$$\begin{array}{rcl}
v: \bP^1 & \ra & \bP^6 \\
\ [t,u]   & \mapsto & [P_1,\ldots,P_7]
\end{array}
$$
and a linear projection given by the $4\times 7$ matrix
$$D=\left( \begin{matrix} 
d_{01} & d_{02} & \ldots & d_{07} \\
 d_{11} & d_{12} & \ldots & d_{17} \\
 d_{21} & d_{22} & \ldots & d_{27} \\
 d_{31} & d_{32} & \ldots & d_{37}  \end{matrix} \right),$$
whose $j$th column consists of the coordinates of $d_j$.  
The condition that $\iota(\bP^1) \subset S$ translates into
\begin{equation} \label{eqn:constant}
F_m(t,u;\sum_j d_{0j}P_j , \ldots,
\sum_j d_{3j}P_j)=0
\end{equation}
for $m=1,2$.  

The perturbation of the matrix $D$
entails rescaling each column of $D$ by a multiplicative scalar,
to first order.  This takes the form
$$D(\epsilon)=\left( d_{ij}(1+s_j\epsilon) \right),$$
keeping in mind that the case $s_1=s_2=\ldots=s_7$
induces a trivial deformation, i.e., rescaling $D$ by a constant.  
The condition (\ref{eqn:constraint}) is automatically satisfied.

We analyze the condition
$$\iota(\epsilon)(\bP^1) \subset S(\epsilon)$$
to first order.  This can be written
$$
\begin{array}{l}
F_m(t,u;\sum_j d_{0j}(1+\epsilon s_j) 
P_j, \ldots ) + \\
 \epsilon  (L\cdot G_m)(t,u;\sum_j d_{0j}(1+\epsilon s_j) P_j, \ldots)=0
\end{array}$$
for $m=1,2$.  
Writing $x_0,x_1,x_2,x_3$ for the homogeneous coordinates on $\bP^3$
and extracting the first derivative---the constant term in $\epsilon$
is zero by (\ref{eqn:constant})---we obtain
$$\sum_{k=0}^3 \frac{\partial F_m}{\partial x_k}
\cdot \sum_j d_{0j}s_jP_j
+
L\cdot G_m=0.$$
Here we should regard $\frac{\partial F_m}{\partial x_k}$ as 
a homogeneous form of degree $7$ in $\{t,u\}$, $L$ also
of degree $7$, and $G_i$ of degree $6$.  Thus we obtain 
two homogeneous forms of degree $13$ in $\{t,u\}$, with each
coefficient linear in the variables $s_1,\ldots,s_7$ 
(and vanishing where $s_1=\cdots=s_7$) and the
eight coefficients of $G_1$ and $G_2$.  
Note, however, that
$$F_m(t,u;d_j)=F_m(t,u;\iota(\epsilon)(r_j))=0,
\quad j=1,\ldots,7,$$
thus $F_m(t,u;\iota(\epsilon)([t,u]))$ is a multiple of 
$\prod_{i=1,\ldots,7} (t-r_iu)$; the same is
true for $L(\iota(\epsilon)([t,u]))$.  
Dividing out by this, we obtain two homogeneous forms of degree $6$
in $\{t,u\}$, with each coefficient linear in the variables;
thus we obtain $14$ linear equations in $14$ independent variables.  
Again, the equations will vanish if $s_1=\ldots=s_7$
so these depend on only six parameters.

\subsection{A unirationality result}
\label{SS:uni}
To execute the computations outlined in Section~\ref{SS:OSS},
we must evaluate all the relevant terms in a specific example.
In practice, finding such an example is much easier when the 
underlying parameter spaces are unirational.   

\begin{prop}
Consider the Hilbert scheme parametrizing the following data:
\begin{itemize}
\item{points $r_1,\ldots,r_7 \in \bP^1$}
\item{a rational curve $C_2 \subset \bP^1 \times \bP^3$
of bidegree $(0,1)$ (i.e., a line);}
\item{a hyperplane section $\{L=0\} \subset \bP^1 \times \bP^3$ containing
$C_2$;}
\item{a morphism $\iota:\bP^1 \ra \bP^1 \times \bP^3$ of bidegree
$(1,6)$ with $\iota(\{r_1,\ldots,r_7\}) \subset \{L=0\}$;}
\item{a K3 surface $S\subset \bP^1 \times \bP^3$, given as a complete
intersection of forms of degree degree $(1,2)$ containing
$C_2$ and $R':=\iota(\bP^1)$.}
\end{itemize}
This space is rational.
\end{prop}
Below, let $f$ be induced from the hyperplane class of $\bP^3$
and $C_1\cup C_2$ be cut out by $L=0$.  Note that $C_1$ is residual
to $C_2$ in the hyperplane $D=\{L=0\}$.  
\begin{coro}
Consider the moduli space of lattice polarized K3 surfaces $S$ of type
$$\begin{array}{r|cccc}
 & f & C_1 & C_2 & R' \\
\hline
f & 4 & 7 & 1 & 6 \\
C_1 & 7 & 8 & 3 & 7 \\
C_2 & 1 & 3 & -2 & 0 \\
R' & 6 & 7 & 0 & -2
\end{array},
$$
equipped with an ordering of the points of $C_1\cap R'$.
This space is unirational.  
\end{coro}
\begin{proof}
The construction is step-by-step:  Seven ordered points on $\bP^1$ are
parametrized by a rational variety.  The lines $C_2$ are as well, i.e.,
the product $\bP^1 \times \mathrm{Gr}(2,4)$.  For each such line,
it is a linear condition for a hypersurface $\{L=0\}$ to vanish along
the line.  Once we have $C_2$ and $\iota$, the K3 surfaces containing
$C_2 \cup R'$ are given by the Grassmannian $\mathrm{Gr}(2,I_{C_2\cup R'}(1,2))$.

Thus we only have to worry about the choice of $\iota$.  We interpret
this as a section of the projection $\bP^1 \times \bP^3 \ra \bP^1$,
given by a collection of sextic polynomials
$$[\iota_0(t,u),\ldots,\iota_3(t,u)],$$
parametrized by a dense open subset of the projective space $\bP^{27}$
on the coefficients.  
After a linear change of coordinates on $\bP^3$, we may assume 
$L=tx_0-ux_1$.  Then the conditions $\iota(r_j) \in \{L=0\}$ take the form
$$r_j\iota_0(r_j,1)-\iota_1(r_j,1),$$
i.e., seven linear equations on the coefficient space.  The resulting
parameter space is thus rational.
\end{proof}

\subsection{Concrete example}
\label{SS:Example}

We exhibit a specific example over $\bQ$ where the morphism $\varphi_{\Lambda}$ is unramified.
Its construction closely follows the unirationality proof in Section~\ref{SS:uni}.

The relevant data is
\begin{itemize}
\item{points $r_1=[0,1],r_2=[1,1],r_3=[-1,1],r_4=[2,1],r_5=[-2,1],r_6=[1/2,1],r_7=[-1/2,1] \in \bP^1$;}
\item{a rational curve $C_2=\{u=x_0=x_3=0\}$;}
\item{a hyperplane section $L=tx_0-ux_1 \subset \bP^1 \times \bP^3$, containing $C_2$;}
\item{a morphism $\iota:\bP^1 \ra \bP^1 \times \bP^3$ given by
$$\begin{array}{rcl}
x_0& =& 4t^6+t^5u+u^6\\
x_1&=&t^6+21t^5u-21t^3u^3+5tu^5 \\
x_2&=&t^6 \\
x_3&=&t^6+t^5u+t^4u^2+t^3u^3+t^2u^4+tu^5+u^6;
\end{array}
$$
} 
\item{a K3 surface $S\subset \bP^1 \times \bP^3$, given as a complete intersection 
$$\{F_1=F_2=0\}$$
where
$$\begin{array}{rcl}
F_1 &= & u(-36134306460 x_0^2 + 1648259021 x_0 x_1 + 179920405271 x_0 x_2 \\
& & + 72385436466 x_0 x_3
 + 49839426 x_1^2 - 3784378416 x_1 x_2\\
& & - 2345703360 x_1 x_3 - 181391061852 x_2^2 - 225811403454 x_2x_3\\
& & - 36251130006 x_3^2) + t(-13678895854 x_0^2 + 671675907 x_0 x_1 \\
& & + 56417839468 x_0x_2 - 8926222977 x_0x_3 + 26209164072x_3^2)
\end{array}
$$
and 
$$\begin{array}{rcl}
F_2 &=&
u(3638964 x_0^2 - 1272831 x_0 x_1 + 29670963 x_0x_2\\
& & - 13458270x_0 x_3 - 22974x_1^2 + 3555552 x_1 x_2\\
& & + 1904792 x_1x_3 - 114701748 x_2^2 - 4837990 x_2 x_3\\
& & + 9819306 x_3^2) + t(12731586 x_0^2 - 575505 x_0 x_1\\
& & -52280172 x_0x_2 - 22071733 x_0 x_3 + 96004264x_2x_3).
\end{array}
$$
}
\end{itemize}
The linear projection matrix is given by
$$D=\left( \begin{matrix}
1 & 6 & 4 & 289 & 225 & 35/32 & 33/32 \\
0 & 6 & -4 & 578 & -450 & 35/64 & -33/64 \\
0 & 1 & 1 & 64 & 64 & 1/64 & 1/64 \\
1 & 7 & 1 & 127 & 43 & 127/64 & 43/64
\end{matrix} \right).$$

For the tangent space computation, write
$$\begin{array}{rcl}
G_1&=&g_{10}x_0+g_{11}x_1+g_{12}x_2+g_{13}x_3 \\
G_2&=&g_{20}x_0+g_{21}x_1+g_{22}x_2+g_{23}x_3
\end{array}$$
and extract the terms linear in $\epsilon$ in
$$
F_m(t,u;\sum_j d_{0j}(1+\epsilon s_j) 
P_j, \ldots )+ \\
 \epsilon  (L\cdot G_m)(t,u;\sum_j d_{0j}(1+\epsilon s_j) P_j, \ldots)=0.
$$
This yields two homogeneous forms in $s$ and $u$ of degree $13$, with coefficients
linear in the $g_{ik}$ and $s_j$, each divisible by
$$\prod_j(t-r_ju)=t(t-u)(t+u)(t-2u)(t+2u)(t-u/2)(t+u/2).$$
Dividing out, we obtain two forms in $s$ and $u$ of degree $6$ with coefficients
linear in the $g_{ik}$ and $s_j$.  We will not reproduce these coefficients here,
but after Gaussian elimination we are left with the system
$$\begin{array}{c}
g_{10}=g_{11}=g_{12}=g_{13}=g_{20}=g_{21}=g_{22}=g_{23}=0 \\
s_1=s_2=s_3=s_4=s_5=s_6=s_7.
\end{array}
$$
Thus the space of infinitesimal deformations of $S$ containing our
curves is trivial, hence our morphism $\varphi_{\Lambda}$ is unramified at this point.

\section{A related construction}
\label{sect:related}
Consider the moduli space $\cF_{\Lambda''}$ of lattice-polarized K3 surfaces of type
\begin{equation} \label{eqn:LP49}
\Lambda'':=\begin{array}{r|cc}
        & h & R' \\
\hline
h   & 12 & 5 \\
R'  & 5 & -2 
\end{array}
\end{equation}
and the moduli space $\cP_{\Lambda''}$ parametrizing pairs $(S,D)$ with
$S\in \cF_{\Lambda''}$ and $D \in |h|$ smooth and transverse to $R'$,
which has dimension 25.  We consider
$$
\widehat{\varphi_{\Lambda''}}:\cP_{\Lambda''}  \ra (\cM_{7,5} \times \cM_{0,5})/\fS_5 
$$
to a variety of dimension $18+5+2=25$.
\begin{prop}
$\widehat{\varphi_{\Lambda''}}$ is birational.
\end{prop}

\begin{proof}
We construct the inverse mapping.
Fix $(C,c_1,\ldots,c_5)$ a generic five-pointed curve of genus seven and
$(\bP^1,a_1,\ldots,a_5)$ a five-pointed curve of genus zero.  
Realize $C$ as
a linear section in $\OGr(5,10)$.  We seek a curve $R'$ arising as a
rational normal quintic curve $\varpi:\bP^1 \ra \OGr(5,10)$ 
with $\varpi(a_i)=c_i,
i=1,\ldots,5$.  The K3 surface $S$ would then arise as the intersection
$$\mathrm{span}(C\cup R')\cap \OGr(5,10).$$

The construction of $\varpi$ boils down to an enumerative
computation by Andrew Kresch:
\begin{lemm}
Given generic points $\Lambda_1,\cdots,\Lambda_5 \in \OGr(5,10)$,
and generic points $a_1,\ldots,a_5 \in \bP^1$, there exists
a unique morphism $\varpi:\bP^1 \ra \OGr(5,10)$ with image of degree five
such that $\varpi(a_i)=\Lambda_i$ for $i=1,\ldots,5$.
\end{lemm}
We have the tautological diagram
$$\begin{array}{ccc}
\oM_{0,5}(\OGr(5,10),5) & \stackrel{\ev^5}{\ra} &  \OGr(5,10)^5 \\
{\scriptstyle \phi} \downarrow \ &    &  \\
\oM_{0,5} &      &
\end{array}
$$
where the $\phi$ is the forgetting morphism and $\ev^5$ is
evaluation at all five marked points.
We seek to show that
$$\deg(\ev_1^*[\pt]\cdots \ev_5^*[\pt] \cdot \phi^*(\pt))=1.$$

We specialize the point in $\oM_{0,5}$ to the configuration of Figure~\ref{fig:FivePoints}.
\begin{figure} 
\includegraphics[scale=.7]{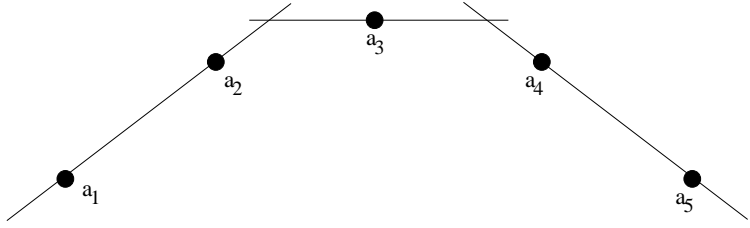}
\caption{A stable five-pointed curve}
\label{fig:FivePoints}
\end{figure}

We assert that 
\begin{enumerate}
\item
there is no line on $\OGr(5,10)$ through two general
points
\item the closure of the union of all conics through two
general points is a Schubert variety parametrizing spaces containing a
particular one-dimensional isotropic vector space $\Xi$ of the ambient
ten dimensional space.
\item there is no conic through three general points.
\end{enumerate}
For the first statement, note that the space of lines is parametrized
by three dimensional isotropic subspaces of the ambient ten dimensional space.
These are contained in the intersection of all the maximal
isotropic subspaces parametrized by the line \cite[Example 4.12]{LM}.
Thus the space of lines has dimension $15$, which is incompatible with
any two points containing a line.

For the second statement,
consider the conics containing $\Lambda_1$ and $\Lambda_2$;
set $\Xi:=\Lambda_1 \cap \Lambda_2$, which is one dimensional,
and note that $\Xi^{\perp}=\Lambda_1+\Lambda_2$.
The $\Lambda \in \OGr(5,10)$ satisfying
$$\Xi \subset \Lambda \subset \Xi^{\perp}$$
are parametrized by $\OGr(4,8)$.  Recall that $\OGr(4,8)$
is just a quadric sixfold by triality; the conics through
two generic points are parametrized by the $\bP^5$ parametrizing
planes containing those two points.  Altogether,
these trace out the full $\OGr(4,8)$ mentioned above.

The last statement follows immediately, as the union of the
conics through two points is a proper subvariety of $\OGr(5,10)$.

These observations imply
that the `end components' of our stable curve
map to conics and the `middle component' maps to a line.
So we are reduced to:
\begin{quote}
Given
$$\Xi = \Lambda_1 \cap \Lambda_2, \quad
\Xi'= \Lambda_4 \cap \Lambda_5$$
how many lines on $\OGr(5,10)$ are incident to
\begin{itemize}
\item
the Schubert variety of $\Lambda  \in \OGr(5,10)$ satisfying $\Xi \subset \Lambda$;
\item
the Schubert variety of $\Lambda \in \OGr(5,10)$ satisfying $\Xi' \subset \Lambda$;
\item
the point $\Lambda_3$?
\end{itemize}
\end{quote}

We know that lines on $\OGr(5,10)$ are in bijection with isotropic three dimensional
spaces, so for an isotropic three dimensional $W$ we must have:
\begin{enumerate}
\item[(a)]
$W$ can be extended to a maximal isotropic space containing $\Xi$;
\item[(b)]
$W$ can be extended to a maximal isotropic space containing $\Xi'$;
\item[(c)]
$W \subset \Lambda_3$.
\end{enumerate}

We deduce from (a) that $W \subset \Xi^{\perp}$ and from (b) that $W \subset (\Xi')^{\perp}$.
It follows that
$$W \subset \Xi^{\perp} \cap (\Xi')^{\perp}  \cap \Lambda_3.$$
Since $\Lambda_3$
is general, this intersection is already three dimensional so that $W$ is uniquely
determined and the answer to our enumerative problem is $1$.
\end{proof}

\bibliographystyle{alpha}
\bibliography{embeddings}

\newpage

\begin{center}
{\Large Appendix}
\end{center}

\

\begin{center}
\noindent
{\Large{71 rational septics through 7 general points on $\OGr(5,10)$}}
\end{center}

\

\medskip

\section*{Overview and generalities}
The orthogonal Grassmannian $\OG(n,2n)$ is, by definition,
one component of the space of
$n$-dimensional subspaces of $V=\C^{2n}$, isotropic for a given nondegenerate
symmetric bilinear form on $V$.
By convention we fix the standard bilinear form, for which
$\langle v,w\rangle=\sum_{i=1}^{2n} v_i w_{2n+1-i}$,
and we take $\OG=\OGr(n,2n)$ to be the component containing
$\langle e_1,\ldots,e_n\rangle$.
Then $\OG$ is a homogeneous projective variety of dimension $n(n-1)/2$,
with $\deg c_1(\OG)=2(n-1)$.
It follows that the space of rational degree $d$ curves has dimension
$n(n-1)/2-3+2(n-1)d$,
so for $n=5$ and $d=7$ this is $63=7\cdot 9$ and we expect a finite number
of rational septics to pass through $7$ general points.

It is known (cf.\ \cite{FP,lopez}) that the number of
rational curves on $\OG$ satisfying incidence conditions imposed by
Schubert varieties of codimension $\geq 2$ in general position
is equal to the corresponding Gromov-Witten invariant:
$$
\#\left\{ 
\begin{array}{l}\text{degree $d$ rational curves through} \\
\text{general translates of $X_{\lambda^1}$, $\dots$, $X_{\lambda^m}$}  
\end{array}
\right\} = I_d([X_{\lambda^1}],\ldots,[X_{\lambda^m}])
$$
for $|\lambda^i|\geq 2$, $\sum |\lambda^i|=n(n-1)/2-3+2(n-1)d+m$, where
$I_d([X_{\lambda^1}],\ldots,[X_{\lambda^m}])$ denotes the Gromov-Witten invariant
\[\int_{\overline{M}_{0,m}(\OG,d)} \mathrm{ev}_1^*[X_{\lambda^1}]\cdots
\mathrm{ev}_m^*[X_{\lambda^m}].\]
The space $\overline{M}_{0,m}(\OG,d)$ is Kontsevich's moduli space of
stable maps of genus zero $m$-marked curves to $\OG$ in degree $d$
(see \cite{KM}), and it comes with $m$ evaluation maps (at the marked points)
$\mathrm{ev}_1$, $\dots$, $\mathrm{ev}_m$ to $\OG$.
The Schubert varieties in $\OG$ are denoted $X_\lambda$, indexed by
strict partitions $\lambda$ whose parts are $<n$.
(A partition is called strict if has no repeated parts.)
The codimension of $X_\lambda$ is equal to $|\lambda|$, the sum of the parts
of $\lambda$.
The article \cite{KT}, which includes a determination of the
$m=3$ invariants, gives a geometric description of
$\OG$ including the Schubert varieties.

\section*{Line numbers}

The space of lines on $\OG$ is known;
see \cite[Example 4.12]{LM}.
It is itself a projective homogeneous variety,
the space $\OGr(n-2,2n)$ of isotropic $(n-2)$-dimensional subspaces of $V=\C^{2n}$.
Therefore the computation of the
$I_1([X_{\lambda^1}],\ldots,[X_{\lambda^m}])$ reduces to the problem
of computing intersection numbers on this homogeneous variety.

There is a well-developed theory using divided difference
operators on polynomials
for performing computations in the cohomology rings of
projective homogeneous varieties of linear algebraic groups,
due to Bernstein, Gelfand, and Gelfand \cite{BGG}
and Demazure \cite{demazure}.
In the setting of the orthogonal flag variety
$\mathrm{OF}(2n)$, parametrizing a space in $\OG$ together with a complete flag of
subspaces, this has been worked out explicitly by
Billey and Haiman \cite{BH}.
It leads to an explicit formula for the Gromov-Witten invariants counting
lines on $\OG$ satisfying incidence conditions with respect to
Schubert varieties in general position.
The formula uses the Schur $P$-polynomials
$P_\lambda=P_\lambda(X)$ indexed by strict partitions $\lambda$,
which form a $\Q$-basis for the ring
$\Q[p_1,p_3,\ldots]$ generated by the odd power sums
$p_k=p_k(X)=x_1^k+x_2^k+\ldots$
(cf.\ Proposition 3.1 of op.\ cit.).
Following op.\ cit.,
to these we associate polynomials in $z_1$, $\dots$, $z_n$, which we will denote
by $P_\lambda(z_1,\ldots,z_n)$, by sending $p_k(X)$
to $-(1/2)(z_1^k+\cdots+z_n^k)$.

\begin{propo}
Introduce the divided difference operators on $\Q[z_1,\ldots,z_n]$:
$$
\begin{array}{cl}
\partial_if&=
\frac{f(z_1,\ldots,z_n)-f(z_1,\ldots,z_{i+1},z_i,\ldots,z_n)}{z_i-z_{i+1}},\\
\partial_{\hat 1} f&=
\frac{f(z_1,\ldots,z_n)-f(-z_2,-z_1,z_3,\ldots,z_n)}{-z_1-z_2},
\end{array}
$$
and for $i\leq j$ let $\partial_{i\ldots j}$ denote
$\partial_i\partial_{i+1}\ldots\partial_j$ and let
$\partial_{j\ldots i}$ denote
$\partial_j\ldots\partial_i$.
Then for any $m$ and $\lambda^1$, $\dots$, $\lambda^m$ satisfying
$|\lambda^i|\geq 2$, $\sum |\lambda^i|=n(n-1)/2-3+2(n-1)+m$,
if we set
\[F=\prod_{i=1}^m \partial_{\hat 1} P_{\lambda^i}(z_1,\ldots,z_n)\]
with the above convention on $P$-polynomials,
then we have
\begin{align*}
I_1&([X_{\lambda^1}],\ldots,[X_{\lambda^m}])=\\
&
\begin{cases}
\partial_{2\ldots n-1}
\partial_{1\ldots n-2}
\partial_{\hat 1}
\cdots
\partial_{2\ldots 3}
\partial_{1\ldots 2}
\partial_{\hat 1}
\partial_{n-2\ldots 1}
\partial_{n-1\ldots 2}F& \text{if $n$ is even},\\
\partial_{2\ldots n-1}
\partial_{\hat 1}
\partial_{2\ldots n-2}
\partial_{1\ldots n-3}
\partial_{\hat 1}
\cdots
\partial_{2\ldots 3}
\partial_{1\ldots 2}
\partial_{\hat 1}
\partial_{n-2\ldots 1}
\partial_{n-1\ldots 2}F& \text{if $n$ is odd},
\end{cases}
\end{align*}
where the $\cdots$ stand for compositions of operators in which the
upper limits of the indices are successively decreased by $2$.
\end{propo}

\begin{proof}
According to Theorem 4 of op.\ cit., if we work with countably many $z$
variables and follow the above convention for associating a symmetric polynomial
in these to a $P$-polynomial $P_\lambda=P_\lambda(X)$, then
$\partial_{\bar 1}P_\lambda$
represents the cycle class of the space of lines incident to $X_\lambda$, and
the displayed composition of divided operators sends the
polynomial representing the class of a point on the space of
lines on $\OG$ to $1$.
So the proposition follows from the observation
that the computation may be performed in the polynomial ring
$\Q[z_1,\ldots,z_n]$.
\end{proof}

When $n=5$, there are $1071$ Gromov-Witten numbers
$I_1([X_{\lambda^1}],\ldots,[X_{\lambda^m}])$, which we take as known
in what follows.

\begin{exam}
One of these numbers counts the number of lines incident to
$15$ general translates of $X_2$
(the codimension-$2$ Schubert variety
of spaces in $\OG$ meeting a given isotropic $3$-dimensional space
nontrivially).
We have $P_2(X)=p_1^2(X)$ sent to $(1/4)(z_1+\cdots+z_5)^2$, which
upon applying $\partial_{\hat 1}$ yields
$-z_3-z_4-z_5$.
We evaluate
\[
\partial_2\partial_3\partial_4
\partial_{\hat 1}
\partial_2\partial_3\partial_1\partial_2
\partial_{\hat 1}
\partial_3\partial_2\partial_1\partial_4\partial_3\partial_2
(-z_3-z_4-z_5)^{15}
\]
and find
\[
I_1([X_2],\ldots,[X_2])=240240.
\]
\end{exam}

We list a few more such numbers:
\begin{equation}
\label{eqn:somelinenumbers}
\begin{aligned}
I_1([X_2],[X_3],[X_{421}],[X_{421}])&=2,&
I_1([X_2],[X_{21}],[X_{421}],[X_{421}])&=2,\\
I_1([X_2],[X_{42}],[X_{4321}])&=1,&
I_1([X_2],[X_{321}],[X_{4321}])&=0,\\
I_1([X_2],[X_{421}],[X_{432}])&=1,&
I_1([X_3],[X_{421}],[X_{431}])&=1,\\
I_1([X_{21}],[X_{421}],[X_{431}])&=1,&
I_1([X_4],[X_{421}],[X_{421}])&=0,\\
I_1([X_{31}],[X_{421}],[X_{421}])&=1,&
I_1([X_{43}],[X_{4321}])&=1,\\
I_1([X_{421}],[X_{4321}])&=0,&
I_1([X_{431}],[X_{432}])&=1.
\end{aligned}
\end{equation}

\section*{Conic numbers}

The \emph{associativity relations} of quantum cohomology
(also known as WDVV equations) are a system of polynomial relations in
Gromov-Witten invariants, which can be used to deduce new invariants from
known ones.
We recall the statement, as formulated in \cite[Eqn.\ (3.3)]{KM},
for the case of $\OG$.
First, the Poincar\'e duality involution $\lambda\mapsto \lambda^\vee$
on the set of partitions indexing the Schubert classes of $\OG$
(basis of the classical cohomology ring of $\OG$), is such that
the set of parts of $\lambda^\vee$ is the complement in
$\{1,\ldots,n-1\}$ of the set of parts of $\lambda$.
We have focused on Gromov-Witten invariants involving
Schubert classes of codimension $\geq 2$ above, because the ones
with fundamental or divisor classes reduce to these by the following
identities:
\begin{gather*}
I_0([X_\lambda],[X_\mu],[X_\nu])=\int_{\OG} [X_\lambda]\cdot [X_\mu]\cdot [X _\nu], \\
I_0([X_{\lambda^1}],\ldots,[X_{\lambda^m}])=0\text{ for $m\ne 3$},
\end{gather*}
and for $d\geq 1$,
\[
I_d([X_{\lambda^1}],\ldots,[X_{\lambda^m}],[X_1])=
dI_d([X_{\lambda^1}],\ldots,[X_{\lambda^m}]).
\]

Since it is needed for the discussion that follows, we record
in Table \ref{table:OGmult} a portion of the
multiplication table for the Schubert classes $\tau_\lambda=[X_\lambda]$
in the classical cohomology ring.
(One can produce this, e.g., using the Pieri formula of \cite{HB}.)

\begin{table}
\[
\begin{array}{c|cccccccc}
&\tau_4&\tau_{31}&\tau_{41}&\tau_{32}&\tau_{42}&\tau_{321}&\tau_{43}&\tau_{421}\\
\hline
\tau_1&\tau_{41}&\tau_{41}+\tau_{32}&\tau_{42}&\tau_{42}+\tau_{321}&\tau_{43}+\tau_{421}&\tau_{421}&
\tau_{431}&\tau_{431}\\
\tau_2&\tau_{42}&2\tau_{42}+\tau_{321}&\tau_{43}+\tau_{421}&\tau_{43}+2\tau_{421}&2\tau_{431}&\tau_{431}&
\tau_{432}&\tau_{432}\\
\tau_3&\tau_{43}&\tau_{43}+2\tau_{421}&\tau_{431}&2\tau_{431}&\tau_{432}&\tau_{432}&0&\tau_{4321}\\
\tau_{21}&\tau_{421}&\tau_{43}+\tau_{421}&\tau_{431}&\tau_{431}&\tau_{432}&0&\tau_{4321}&0\\
\tau_4&0&\tau_{431}&0&\tau_{432}&0&\tau_{4321}&0&0\\
\tau_{31}&\tau_{431}&2\tau_{431}&\tau_{432}&\tau_{432}&\tau_{4321}&0&0&0
\end{array}
\]
\caption{Portion of multiplication table for $H^*(\OGr(5,10))$}
\label{table:OGmult}
\end{table}


Given $d\geq 1$, $m\geq 4$ and $\lambda^1$, $\dots$, $\lambda^m$ satisfying
$$
|\lambda^i|\geq 1, \quad\quad  \sum |\lambda^i|=n(n-1)/2-4+2d(n-1)+m,
$$
the corresponding associativity relation reads
\begin{align}
\begin{split}
\label{eqn:assocrel}
&\sum_{d',\mu,A}
I_{d'}(\tau_{\lambda^{i_1}},\ldots,\tau_{\lambda^{i_a}},
\tau_{\lambda^{m-3}},\tau_{\lambda^{m-2}},\tau_\mu)
I_{d-d'}(\tau_{\lambda^{j_1}},\ldots,\tau_{\lambda^{j_b}},
\tau_{\lambda^{m-1}},\tau_{\lambda^m},\tau_{\mu^\vee})\\
&=\sum_{d',\mu,A}
I_{d'}(\tau_{\lambda^{i_1}},\ldots,\tau_{\lambda^{i_a}},
\tau_{\lambda^{m-3}},\tau_{\lambda^m},\tau_\mu)
I_{d-d'}(\tau_{\lambda^{j_1}},\ldots,\tau_{\lambda^{j_b}},
\tau_{\lambda^{m-2}},\tau_{\lambda^{m-1}},\tau_{\mu^\vee}),
\end{split}
\end{align}
where the first, respectively second sum is over integers $0\leq d'\leq d$,
strict partitions $\mu$ with parts less than $n$,
and subsets
$A\subset \{1,\ldots,m-4\}$
such that
\begin{equation}
\label{eqn:assocrelcond}
\sum_{i\in A\cup\{m-3,m-2\}} |\lambda^i|+|\mu|=n(n-1)/2+2d'(n-1)+a,
\end{equation}
respectively the same condition with $m-2$ replaced by $m$.
In the equations \eqref{eqn:assocrel}--\eqref{eqn:assocrelcond}
$a$, respectively $b$ denotes the cardinality of $A$, respectively
$B:=\{1,\ldots,m-4\}\smallsetminus A$,
and we write $A=\{i_1,\ldots,i_a\}$ and $B=\{j_1,\ldots,j_b\}$.

In case $d=2$ in \eqref{eqn:assocrel} we observe the following:
(i) all terms with $d'=1$, and hence $d-d'=1$, are known by the
previous section;
(ii) terms with $d'=0$ contribute
\begin{equation}
\label{eqn:lhs0}
\sum_\mu \Big(\int_{\OG}\tau_{\lambda^{m-3}}\tau_{\lambda^{m-2}}\tau_{\mu^\vee}\Big)
I_2(\tau_{\lambda^1},\ldots,\tau_{\lambda^{m-4}},
\tau_{\lambda^{m-1}},\tau_{\lambda^m},\tau_\mu)
\end{equation}
to the left-hand side and
\begin{equation}
\label{eqn:rhs0}
\sum_\mu \Big(\int_{\OG}\tau_{\lambda^{m-3}}\tau_{\lambda^m}\tau_{\mu^\vee}\Big)
I_2(\tau_{\lambda^1},\ldots,\tau_{\lambda^{m-4}},
\tau_{\lambda^{m-2}},\tau_{\lambda^{m-1}},\tau_\mu)
\end{equation}
to the right-hand side;
(iii) terms with $d'=2$ contribute
\begin{equation}
\label{eqn:lhs2}
\sum_\mu \Big(\int_{\OG}\tau_{\lambda^{m-1}}\tau_{\lambda^m}\tau_{\mu^\vee}\Big)
I_2(\tau_{\lambda^1},\ldots,\tau_{\lambda^{m-4}},
\tau_{\lambda^{m-3}},\tau_{\lambda^{m-2}},\tau_\mu)
\end{equation}
to the left-hand side and
\begin{equation}
\label{eqn:rhs2}
\sum_\mu \Big(\int_{\OG}\tau_{\lambda^{m-2}}\tau_{\lambda^{m-1}}\tau_{\mu^\vee}\Big)
I_2(\tau_{\lambda^1},\ldots,\tau_{\lambda^{m-4}},
\tau_{\lambda^{m-3}},\tau_{\lambda^m},\tau_\mu)
\end{equation}
to the right-hand side.

Now it is clear that the associativity relations determine many of the
Gromov-Witten numbers $I_2(\tau_{\lambda^1},\ldots,\tau_{\lambda^m})$.
We spell out the cases of interest, and for each case we will subsequently
take the corresponding Gromov-Witten numbers as known.
Notice that we always take $d=2$ in the following applications of \eqref{eqn:assocrel}.

\emph{Case 1.} Two point conditions:
$I_2(\ldots,\tau_{4321},\tau_{4321})$.
We apply \eqref{eqn:assocrel} with
$\lambda^{m-3}=1$, $\lambda^{m-2}=432$, $|\lambda^{m-1}|\geq 2$,
$\lambda^m=4321$.
Then \eqref{eqn:rhs0}--\eqref{eqn:rhs2} vanish,
while
\eqref{eqn:lhs0} contributes
$I_2(\tau_{\lambda^1},\ldots,\tau_{\lambda^{m-4}},\tau_{\lambda^{m-1}},
\tau_{4321},\tau_{4321})$.

\emph{Case 2.} Point and line conditions:
$I_2(\ldots,\tau_{432},\tau_{4321})$.
We apply \eqref{eqn:assocrel} with
$\lambda^{m-3}=1$, $\lambda^{m-2}=431$, $|\lambda^{m-1}|\geq 2$,
$\lambda^m=4321$.
Then \eqref{eqn:rhs0}--\eqref{eqn:lhs2} vanish,
\eqref{eqn:rhs2} either vanishes or is known by Case 1,
and \eqref{eqn:lhs0} contributes
$I_2(\tau_{\lambda^1},\ldots,\tau_{\lambda^{m-4}},\tau_{\lambda^{m-1}},
\tau_{432},\tau_{4321})$.

\emph{Case 3.} Point and plane:
$I_2(\ldots,\tau_{431},\tau_{4321})$.
We apply \eqref{eqn:assocrel} with
$\lambda^{m-3}=1$, $\lambda^{m-2}=421$, $|\lambda^{m-1}|\geq 2$,
$\lambda^m=4321$.
Then \eqref{eqn:rhs0}--\eqref{eqn:lhs2} vanish,
\eqref{eqn:rhs2} either vanishes or is known by previous cases,
and \eqref{eqn:lhs0} contributes
$I_2(\tau_{\lambda^1},\ldots,\tau_{\lambda^{m-4}},\tau_{\lambda^{m-1}},
\tau_{431},\tau_{4321})$.

\emph{Case 4.} Point and $X_{421}$:
$I_2(\ldots,\tau_{421},\tau_{4321})$.
We apply \eqref{eqn:assocrel} with
$\lambda^{m-3}=1$, $\lambda^{m-2}=321$, $|\lambda^{m-1}|\geq 2$,
$\lambda^m=4321$, and proceed as in Case 3.

\emph{Case 5.} Point and $X_{43}$:
$I_2(\ldots,\tau_{43},\tau_{4321})$.
We apply \eqref{eqn:assocrel} with
$\lambda^{m-3}=1$, $\lambda^{m-2}=42$, $|\lambda^{m-1}|\geq 2$,
$\lambda^m=4321$.
Then \eqref{eqn:rhs0}--\eqref{eqn:lhs2} vanish,
\eqref{eqn:rhs2} either vanishes or is known by previous cases,
and \eqref{eqn:lhs0} contributes
$I_2(\tau_{\lambda^1},\ldots,\tau_{\lambda^{m-4}},\tau_{\lambda^{m-1}},
\tau_{43},\tau_{4321})+
I_2(\tau_{\lambda^1},\ldots,\tau_{\lambda^{m-4}},\tau_{\lambda^{m-1}},
\tau_{421},\tau_{4321})$.

\emph{Case 6.} Point and $X_{42}$:
$I_2(\ldots,\tau_{42},\tau_{4321})$.
We apply \eqref{eqn:assocrel} with
$\lambda^{m-3}=1$, $\lambda^{m-2}=41$, $|\lambda^{m-1}|\geq 2$,
$\lambda^m=4321$, and proceed as in Case 3.

\emph{Case 7.} Point and $X_{321}$:
$I_2(\ldots,\tau_{321},\tau_{4321})$.
We apply \eqref{eqn:assocrel} with
$\lambda^{m-3}=1$, $\lambda^{m-2}=32$, $|\lambda^{m-1}|\geq 2$,
$\lambda^m=4321$, and proceed as in Case 5.

\emph{Case 8.} Point and $X_{41}$:
$I_2(\ldots,\tau_{41},\tau_{4321})$.
We apply \eqref{eqn:assocrel} with
$\lambda^{m-3}=1$, $\lambda^{m-2}=4$, $|\lambda^{m-1}|\geq 2$,
$\lambda^m=4321$, and proceed as in Case 3.

\emph{Case 9.} Point and $X_{32}$:
$I_2(\ldots,\tau_{32},\tau_{4321})$.
We apply \eqref{eqn:assocrel} with
$\lambda^{m-3}=1$, $\lambda^{m-2}=31$, $|\lambda^{m-1}|\geq 2$,
$\lambda^m=4321$, and proceed as in Case 5.

\emph{Case 10.} Two line conditions:
$I_2(\ldots,\tau_{432},\tau_{432})$.
We apply \eqref{eqn:assocrel} with
$\lambda^{m-3}=1$, $\lambda^{m-2}=431$, $|\lambda^{m-1}|\geq 2$,
$\lambda^m=432$.
Then \eqref{eqn:lhs2} vanishes,
\eqref{eqn:rhs2} vanishes or is known by Case 2,
\eqref{eqn:rhs0} contributes
$I_2(\tau_{\lambda^1},\ldots,\tau_{\lambda^{m-4}},\tau_{\lambda^{m-1}},
\tau_{431},\tau_{4321})$,
and \eqref{eqn:lhs0} contributes
$I_2(\tau_{\lambda^1},\ldots,\tau_{\lambda^{m-4}},\tau_{\lambda^{m-1}},
\tau_{432},\tau_{432})$.

\emph{Case 11.} Line and plane:
$I_2(\ldots,\tau_{431},\tau_{432})$.
We apply \eqref{eqn:assocrel} with
$\lambda^{m-3}=1$, $\lambda^{m-2}=421$, $|\lambda^{m-1}|\geq 2$,
$\lambda^m=432$, and proceed as in Case 10.

\emph{Case 12.} Line and $X_{421}$:
$I_2(\ldots,\tau_{421},\tau_{432})$.
We apply \eqref{eqn:assocrel} with
$\lambda^{m-3}=1$, $\lambda^{m-2}=321$, $|\lambda^{m-1}|\geq 2$,
$\lambda^m=432$, and proceed as in Case 10.

\emph{Case 13.} Line and $X_{43}$:
$I_2(\ldots,\tau_{43},\tau_{432})$.
We apply \eqref{eqn:assocrel} with
$\lambda^{m-3}=1$, $\lambda^{m-2}=42$, $|\lambda^{m-1}|\geq 2$,
$\lambda^m=432$.
Then \eqref{eqn:lhs2} vanishes,
\eqref{eqn:rhs2} vanishes or is known by previous cases,
\eqref{eqn:rhs0} contributes
$I_2(\tau_{\lambda^1},\ldots,\tau_{\lambda^{m-4}},\tau_{\lambda^{m-1}},
\tau_{42},\tau_{4321})$,
and \eqref{eqn:lhs0} contributes
$I_2(\tau_{\lambda^1},\ldots,\tau_{\lambda^{m-4}},\tau_{\lambda^{m-1}},
\tau_{43},\tau_{432})+
I_2(\tau_{\lambda^1},\ldots,\tau_{\lambda^{m-4}},\tau_{\lambda^{m-1}},
\tau_{421},\tau_{432})$.

\emph{Case 14.} Line and $X_{42}$:
$I_2(\ldots,\tau_{42},\tau_{432})$.
We apply \eqref{eqn:assocrel} with
$\lambda^{m-3}=1$, $\lambda^{m-2}=41$, $|\lambda^{m-1}|\geq 2$,
$\lambda^m=432$, and proceed as in Case 10.

\emph{Case 15.} Line and $X_{321}$:
$I_2(\ldots,\tau_{321},\tau_{432})$.
We apply \eqref{eqn:assocrel} with
$\lambda^{m-3}=1$, $\lambda^{m-2}=32$, $|\lambda^{m-1}|\geq 2$,
$\lambda^m=432$, and proceed as in Case 13.

\emph{Case 16.} Two plane conditions:
$I_2(\ldots,\tau_{431},\tau_{431})$.
We apply \eqref{eqn:assocrel} with
$\lambda^{m-3}=1$, $\lambda^{m-2}=421$, $|\lambda^{m-1}|\geq 2$,
$\lambda^m=431$.
Then \eqref{eqn:lhs2} vanishes or is known by Case 4,
\eqref{eqn:rhs0} is known by Case 12,
\eqref{eqn:rhs2} vanishes or is known by previous cases, and
\eqref{eqn:lhs0} contributes
$I_2(\tau_{\lambda^1},\ldots,\tau_{\lambda^{m-4}},\tau_{\lambda^{m-1}},
\tau_{431},\tau_{431})$.

\emph{Case 17.} Plane and $X_{421}$:
$I_2(\ldots,\tau_{421},\tau_{431})$.
We apply \eqref{eqn:assocrel} with
$\lambda^{m-3}=1$, $\lambda^{m-2}=321$, $|\lambda^{m-1}|\geq 2$,
$\lambda^m=431$.
Then \eqref{eqn:rhs0} is known by Case 15,
\eqref{eqn:lhs2} and \eqref{eqn:rhs2} vanish or are known by previous cases,
and \eqref{eqn:lhs0} contributes
$I_2(\tau_{\lambda^1},\ldots,\tau_{\lambda^{m-4}},\tau_{\lambda^{m-1}},
\tau_{421},\tau_{431})$.

We list a few of the conic numbers:
\begin{equation}
\label{eqn:someconicnumbers}
\begin{aligned}
I_2(\tau_2,\tau_{421},\tau_{431},\tau_{4321})&=3,&
I_2(\tau_3,\tau_{421},\tau_{431},\tau_{432})&=5,\\
I_2(\tau_{21},\tau_{421},\tau_{431},\tau_{432})&=4,&
I_2(\tau_{421},\tau_{432},\tau_{4321})&=1,\\
I_2(\tau_{431},\tau_{431},\tau_{4321})&=1,&
I_2(\tau_{431},\tau_{432},\tau_{432})&=2,
\end{aligned}
\end{equation}
In total, Cases 1 through 17 determine $1459$ conic numbers.

\begin{exam}
The number $I_2(\tau_2,\tau_{421},\tau_{431},\tau_{4321})$ falls under
Case 3.
We have $m=5$, $\lambda^2=1$, $\lambda^3=421$, $\lambda^5=4321$,
and either
$\lambda^4=2$, hence $\lambda^1=421$ with \eqref{eqn:assocrel} giving
\begin{align*}
I_2(\tau_2,&\tau_{421},\tau_{431},\tau_{4321}) +
I_1(\tau_1,\tau_4,\tau_{421},\tau_{421})I_1(\tau_2,\tau_{321},\tau_{4321})\\
& +
I_1(\tau_1,\tau_{31},\tau_{421},\tau_{421})I_1(\tau_2,\tau_{42},\tau_{4321}) \\
& = I_2(\tau_1,\tau_{421},\tau_{432},\tau_{4321})
+ I_1(\tau_1,\tau_{43},\tau_{4321})I_1(\tau_2,\tau_{21},\tau_{421},\tau_{421}) \\
& + I_1(\tau_1,\tau_{421},\tau_{4321})I_1(\tau_2,\tau_3,\tau_{421},\tau_{421})
+ I_1(\tau_1,\tau_1,\tau_{421},\tau_{4321})I_1(\tau_2,\tau_{421},\tau_{432}),
\end{align*}
or
$\lambda^4=421$, hence $\lambda^1=2$ and \eqref{eqn:assocrel} giving
\begin{align*}
I_2(\tau_2,&\tau_{421},\tau_{431},\tau_{4321}) +
I_1(\tau_1,\tau_2,\tau_{421},\tau_{432})I_1(\tau_1,\tau_{421},\tau_{4321}) \\
& = I_1(\tau_1,\tau_{43},\tau_{4321})I_1(\tau_2,\tau_{21},\tau_{421},\tau_{421})
 + I_1(\tau_1,\tau_{421},\tau_{4321})I_1(\tau_2,\tau_3,\tau_{421},\tau_{421}) \\
& + I_1(\tau_1,\tau_2,\tau_{42},\tau_{4321})I_1(\tau_{31},\tau_{421},\tau_{421})
  + I_1(\tau_1,\tau_2,\tau_{321},\tau_{4321})I_1(\tau_4,\tau_{421},\tau_{421}).
\end{align*}
Either way, we obtain $I_2(\tau_2,\tau_{421},\tau_{431},\tau_{4321})=3$.
One way requires the Case 2 number
$I_2(\tau_{421},\tau_{432},\tau_{4321})$.
The needed line numbers appear in \eqref{eqn:somelinenumbers}.
\end{exam}

\begin{table}
\addtolength{\abovecaptionskip}{-16pt}
\begin{small}
\setlength{\jot}{0pt}
\begin{align*}
I_3(\tau_2,\tau_2,\tau_2,\tau_2,\tau_{4321},\tau_{4321},\tau_{4321})&=81,&
I_3(\tau_2,\tau_2,\tau_2,\tau_3,\tau_{432},\tau_{4321},\tau_{4321})&=216,\\
I_3(\tau_2,\tau_2,\tau_2,\tau_{21},\tau_{432},\tau_{4321},\tau_{4321})&=135,&
I_3(\tau_2,\tau_2,\tau_3,\tau_{4321},\tau_{4321},\tau_{4321})&=18,\\
I_3(\tau_2,\tau_2,\tau_{21},\tau_{4321},\tau_{4321},\tau_{4321})&=9,&
I_3(\tau_2,\tau_2,\tau_4,\tau_{432},\tau_{4321},\tau_{4321})&=24,\\
I_3(\tau_2,\tau_2,\tau_{31},\tau_{432},\tau_{4321},\tau_{4321})&=42,&
I_3(\tau_2,\tau_3,\tau_3,\tau_{432},\tau_{4321},\tau_{4321})&=40,\\
I_3(\tau_2,\tau_3,\tau_{21},\tau_{432},\tau_{4321},\tau_{4321})&=26,&
I_3(\tau_2,\tau_{21},\tau_{21},\tau_{432},\tau_{4321},\tau_{4321})&=16,\\
I_3(\tau_2,\tau_4,\tau_{4321},\tau_{4321},\tau_{4321})&=3,&
I_3(\tau_2,\tau_{31},\tau_{4321},\tau_{4321},\tau_{4321})&=3,\\
I_3(\tau_2,\tau_{41},\tau_{432},\tau_{4321},\tau_{4321})&=7,&
I_3(\tau_2,\tau_{32},\tau_{432},\tau_{4321},\tau_{4321})&=6,\\
I_3(\tau_3,\tau_3,\tau_3,\tau_{431},\tau_{4321},\tau_{4321})&=52,&
I_3(\tau_3,\tau_3,\tau_{21},\tau_{431},\tau_{4321},\tau_{4321})&=36,\\
I_3(\tau_3,\tau_3,\tau_{4321},\tau_{4321},\tau_{4321})&=4,&
I_3(\tau_3,\tau_{21},\tau_{21},\tau_{431},\tau_{4321},\tau_{4321})&=25,\\
I_3(\tau_3,\tau_{21},\tau_{4321},\tau_{4321},\tau_{4321})&=2,&
I_3(\tau_3,\tau_4,\tau_{432},\tau_{4321},\tau_{4321})&=4,\\
I_3(\tau_3,\tau_{31},\tau_{432},\tau_{4321},\tau_{4321})&=8,&
I_3(\tau_3,\tau_{41},\tau_{431},\tau_{4321},\tau_{4321})&=7,\\
I_3(\tau_3,\tau_{32},\tau_{431},\tau_{4321},\tau_{4321})&=10,&
I_3(\tau_{21},\tau_{21},\tau_{21},\tau_{431},\tau_{4321},\tau_{4321})&=17,\\
I_3(\tau_{21},\tau_{21},\tau_{4321},\tau_{4321},\tau_{4321})&=1,&
I_3(\tau_{21},\tau_4,\tau_{432},\tau_{4321},\tau_{4321})&=3,\\
I_3(\tau_{21},\tau_{31},\tau_{432},\tau_{4321},\tau_{4321})&=5,&
I_3(\tau_{21},\tau_{41},\tau_{431},\tau_{4321},\tau_{4321})&=5,\\
I_3(\tau_{21},\tau_{32},\tau_{431},\tau_{4321},\tau_{4321})&=7,&
I_3(\tau_{41},\tau_{4321},\tau_{4321},\tau_{4321})&=1,\\
I_3(\tau_{32},\tau_{4321},\tau_{4321},\tau_{4321})&=0,&
I_3(\tau_{42},\tau_{432},\tau_{4321},\tau_{4321})&=2,\\
I_3(\tau_{321},\tau_{432},\tau_{4321},\tau_{4321})&=0,&
I_3(\tau_{43},\tau_{431},\tau_{4321},\tau_{4321})&=1,\\
I_3(\tau_{421},\tau_{431},\tau_{4321},\tau_{4321})&=2.
\end{align*}
\end{small}
\caption{Degree 3, Case 1 numbers}
\label{table:d3c1}
\end{table}

\begin{table}
\addtolength{\abovecaptionskip}{-16pt}
\begin{small}
\setlength{\jot}{0pt}
\begin{align*}
I_3(\tau_2,\tau_2,\tau_3,\tau_3,\tau_{432},\tau_{432},\tau_{4321})&=548,&
I_3(\tau_2,\tau_2,\tau_3,\tau_{21},\tau_{432},\tau_{432},\tau_{4321})&=379,\\
I_3(\tau_2,\tau_2,\tau_{21},\tau_{21},\tau_{432},\tau_{432},\tau_{4321})&=260,&
I_3(\tau_2,\tau_2,\tau_{41},\tau_{432},\tau_{432},\tau_{4321})&=80,\\
I_3(\tau_2,\tau_2,\tau_{32},\tau_{432},\tau_{432},\tau_{4321})&=105,&
I_3(\tau_2,\tau_3,\tau_3,\tau_3,\tau_{431},\tau_{432},\tau_{4321})&=753,\\
I_3(\tau_2,\tau_3,\tau_3,\tau_{21},\tau_{431},\tau_{432},\tau_{4321})&=531,&
I_3(\tau_2,\tau_3,\tau_{21},\tau_{21},\tau_{431},\tau_{432},\tau_{4321})&=377,\\
I_3(\tau_2,\tau_3,\tau_4,\tau_{432},\tau_{432},\tau_{4321})&=47,&
I_3(\tau_2,\tau_3,\tau_{31},\tau_{432},\tau_{432},\tau_{4321})&=109,\\
I_3(\tau_2,\tau_3,\tau_{41},\tau_{431},\tau_{432},\tau_{4321})&=96,&
I_3(\tau_2,\tau_3,\tau_{32},\tau_{431},\tau_{432},\tau_{4321})&=139,\\
I_3(\tau_2,\tau_{21},\tau_{21},\tau_{21},\tau_{431},\tau_{432},\tau_{4321})&=270,&
I_3(\tau_2,\tau_{21},\tau_4,\tau_{432},\tau_{432},\tau_{4321})&=33,\\
I_3(\tau_2,\tau_{21},\tau_{31},\tau_{432},\tau_{432},\tau_{4321})&=76,&
I_3(\tau_2,\tau_{21},\tau_{41},\tau_{431},\tau_{432},\tau_{4321})&=66,\\
I_3(\tau_2,\tau_{21},\tau_{32},\tau_{431},\tau_{432},\tau_{4321})&=103,&
I_3(\tau_2,\tau_{42},\tau_{432},\tau_{432},\tau_{4321})&=22,\\
I_3(\tau_2,\tau_{321},\tau_{432},\tau_{432},\tau_{4321})&=9,&
I_3(\tau_2,\tau_{43},\tau_{431},\tau_{432},\tau_{4321})&=19,\\
I_3(\tau_2,\tau_{421},\tau_{431},\tau_{432},\tau_{4321})&=21,&
I_3(\tau_3,\tau_3,\tau_3,\tau_{432},\tau_{432},\tau_{4321})&=92,\\
I_3(\tau_3,\tau_3,\tau_{21},\tau_{432},\tau_{432},\tau_{4321})&=64,&
I_3(\tau_3,\tau_3,\tau_4,\tau_{431},\tau_{432},\tau_{4321})&=59,\\
I_3(\tau_3,\tau_3,\tau_{31},\tau_{431},\tau_{432},\tau_{4321})&=142,&
I_3(\tau_3,\tau_{21},\tau_{21},\tau_{432},\tau_{432},\tau_{4321})&=45,\\
I_3(\tau_3,\tau_{21},\tau_4,\tau_{431},\tau_{432},\tau_{4321})&=41,&
I_3(\tau_3,\tau_{21},\tau_{31},\tau_{431},\tau_{432},\tau_{4321})&=101,\\
I_3(\tau_3,\tau_{41},\tau_{432},\tau_{432},\tau_{4321})&=13,&
I_3(\tau_3,\tau_{32},\tau_{432},\tau_{432},\tau_{4321})&=18,\\
I_3(\tau_3,\tau_{42},\tau_{431},\tau_{432},\tau_{4321})&=25,&
I_3(\tau_3,\tau_{321},\tau_{431},\tau_{432},\tau_{4321})&=11,\\
I_3(\tau_{21},\tau_{21},\tau_{21},\tau_{432},\tau_{432},\tau_{4321})&=31,&
I_3(\tau_{21},\tau_{21},\tau_4,\tau_{431},\tau_{432},\tau_{4321})&=28,\\
I_3(\tau_{21},\tau_{21},\tau_{31},\tau_{431},\tau_{432},\tau_{4321})&=73,&
I_3(\tau_{21},\tau_{41},\tau_{432},\tau_{432},\tau_{4321})&=9,\\
I_3(\tau_{21},\tau_{32},\tau_{432},\tau_{432},\tau_{4321})&=13,&
I_3(\tau_{21},\tau_{42},\tau_{431},\tau_{432},\tau_{4321})&=17,\\
I_3(\tau_{21},\tau_{321},\tau_{431},\tau_{432},\tau_{4321})&=10,&
I_3(\tau_4,\tau_4,\tau_{432},\tau_{432},\tau_{4321})&=4,\\
I_3(\tau_4,\tau_{31},\tau_{432},\tau_{432},\tau_{4321})&=9,&
I_3(\tau_4,\tau_{41},\tau_{431},\tau_{432},\tau_{4321})&=8,\\
I_3(\tau_4,\tau_{32},\tau_{431},\tau_{432},\tau_{4321})&=9,&
I_3(\tau_{31},\tau_{31},\tau_{432},\tau_{432},\tau_{4321})&=22,\\
I_3(\tau_{31},\tau_{41},\tau_{431},\tau_{432},\tau_{4321})&=17,&
I_3(\tau_{31},\tau_{32},\tau_{431},\tau_{432},\tau_{4321})&=27,\\
I_3(\tau_{43},\tau_{432},\tau_{432},\tau_{4321})&=3,&
I_3(\tau_{421},\tau_{432},\tau_{432},\tau_{4321})&=3,\\
I_3(\tau_{431},\tau_{431},\tau_{432},\tau_{4321})&=5.
\end{align*}
\end{small}
\caption{Degree 3, Case 2 numbers}
\label{table:d3c2}
\end{table}

\begin{table}
\addtolength{\abovecaptionskip}{-16pt}
\begin{small}
\setlength{\jot}{0pt}
\begin{align*}
I_3(\tau_3,\tau_3,\tau_3,\tau_3,\tau_{431},\tau_{431},\tau_{4321})&=1062,&
I_3(\tau_3,\tau_3,\tau_3,\tau_{21},\tau_{431},\tau_{431},\tau_{4321})&=750,\\
I_3(\tau_3,\tau_3,\tau_{21},\tau_{21},\tau_{431},\tau_{431},\tau_{4321})&=534,&
I_3(\tau_3,\tau_3,\tau_{41},\tau_{431},\tau_{431},\tau_{4321})&=120,\\
I_3(\tau_3,\tau_3,\tau_{32},\tau_{431},\tau_{431},\tau_{4321})&=174,&
I_3(\tau_3,\tau_{21},\tau_{21},\tau_{21},\tau_{431},\tau_{431},\tau_{4321})&=385,\\
I_3(\tau_3,\tau_{21},\tau_{41},\tau_{431},\tau_{431},\tau_{4321})&=83,&
I_3(\tau_3,\tau_{21},\tau_{32},\tau_{431},\tau_{431},\tau_{4321})&=128,\\
I_3(\tau_3,\tau_{43},\tau_{431},\tau_{431},\tau_{4321})&=20,&
I_3(\tau_3,\tau_{421},\tau_{431},\tau_{431},\tau_{4321})&=21,\\
I_3(\tau_{21},\tau_{21},\tau_{21},\tau_{21},\tau_{431},\tau_{431},\tau_{4321})&=282,&
I_3(\tau_{21},\tau_{21},\tau_{41},\tau_{431},\tau_{431},\tau_{4321})&=58,\\
I_3(\tau_{21},\tau_{21},\tau_{32},\tau_{431},\tau_{431},\tau_{4321})&=95,&
I_3(\tau_{21},\tau_{43},\tau_{431},\tau_{431},\tau_{4321})&=13,\\
I_3(\tau_{21},\tau_{421},\tau_{431},\tau_{431},\tau_{4321})&=16,&
I_3(\tau_{41},\tau_{41},\tau_{431},\tau_{431},\tau_{4321})&=12,\\
I_3(\tau_{41},\tau_{32},\tau_{431},\tau_{431},\tau_{4321})&=17,&
I_3(\tau_{32},\tau_{32},\tau_{431},\tau_{431},\tau_{4321})&=28.
\end{align*}
\end{small}
\caption{Degree 3, Case 3 numbers}
\label{table:d3c3}
\end{table}

\begin{table}
\addtolength{\abovecaptionskip}{-16pt}
\begin{small}
\setlength{\jot}{0pt}
\begin{align*}
I_3(\tau_2,\tau_3,\tau_3,\tau_3,\tau_{432},\tau_{432},\tau_{432})&=1416,&
I_3(\tau_2,\tau_3,\tau_3,\tau_{21},\tau_{432},\tau_{432},\tau_{432})&=996,\\
I_3(\tau_2,\tau_3,\tau_{21},\tau_{21},\tau_{432},\tau_{432},\tau_{432})&=708,&
I_3(\tau_2,\tau_3,\tau_{41},\tau_{432},\tau_{432},\tau_{432})&=189,\\
I_3(\tau_2,\tau_3,\tau_{32},\tau_{432},\tau_{432},\tau_{432})&=270,&
I_3(\tau_2,\tau_{21},\tau_{21},\tau_{21},\tau_{432},\tau_{432},\tau_{432})&=510,\\
I_3(\tau_2,\tau_{21},\tau_{41},\tau_{432},\tau_{432},\tau_{432})&=129,&
I_3(\tau_2,\tau_{21},\tau_{32},\tau_{432},\tau_{432},\tau_{432})&=201,\\
I_3(\tau_2,\tau_{43},\tau_{432},\tau_{432},\tau_{432})&=42,&
I_3(\tau_2,\tau_{421},\tau_{432},\tau_{432},\tau_{432})&=42,\\
I_3(\tau_3,\tau_3,\tau_3,\tau_3,\tau_{431},\tau_{432},\tau_{432})&=1940,&
I_3(\tau_3,\tau_3,\tau_3,\tau_{21},\tau_{431},\tau_{432},\tau_{432})&=1362,\\
I_3(\tau_3,\tau_3,\tau_{21},\tau_{21},\tau_{431},\tau_{432},\tau_{432})&=966,&
I_3(\tau_3,\tau_3,\tau_4,\tau_{432},\tau_{432},\tau_{432})&=112,\\
I_3(\tau_3,\tau_3,\tau_{31},\tau_{432},\tau_{432},\tau_{432})&=268,&
I_3(\tau_3,\tau_3,\tau_{41},\tau_{431},\tau_{432},\tau_{432})&=228,\\
I_3(\tau_3,\tau_3,\tau_{32},\tau_{431},\tau_{432},\tau_{432})&=320,&
I_3(\tau_3,\tau_{21},\tau_{21},\tau_{21},\tau_{431},\tau_{432},\tau_{432})&=696,\\
I_3(\tau_3,\tau_{21},\tau_4,\tau_{432},\tau_{432},\tau_{432})&=77,&
I_3(\tau_3,\tau_{21},\tau_{31},\tau_{432},\tau_{432},\tau_{432})&=191,\\
I_3(\tau_3,\tau_{21},\tau_{41},\tau_{431},\tau_{432},\tau_{432})&=156,&
I_3(\tau_3,\tau_{21},\tau_{32},\tau_{431},\tau_{432},\tau_{432})&=236,\\
I_3(\tau_3,\tau_{42},\tau_{432},\tau_{432},\tau_{432})&=50,&
I_3(\tau_3,\tau_{321},\tau_{432},\tau_{432},\tau_{432})&=22,\\
I_3(\tau_3,\tau_{43},\tau_{431},\tau_{432},\tau_{432})&=42,&
I_3(\tau_3,\tau_{421},\tau_{431},\tau_{432},\tau_{432})&=39,\\
I_3(\tau_{21},\tau_{21},\tau_{21},\tau_{21},\tau_{431},\tau_{432},\tau_{432})&=512,&
I_3(\tau_{21},\tau_{21},\tau_4,\tau_{432},\tau_{432},\tau_{432})&=52,\\
I_3(\tau_{21},\tau_{21},\tau_{31},\tau_{432},\tau_{432},\tau_{432})&=139,&
I_3(\tau_{21},\tau_{21},\tau_{41},\tau_{431},\tau_{432},\tau_{432})&=108,\\
I_3(\tau_{21},\tau_{21},\tau_{32},\tau_{431},\tau_{432},\tau_{432})&=176,&
I_3(\tau_{21},\tau_{42},\tau_{432},\tau_{432},\tau_{432})&=34,\\
I_3(\tau_{21},\tau_{321},\tau_{432},\tau_{432},\tau_{432})&=20,&
I_3(\tau_{21},\tau_{43},\tau_{431},\tau_{432},\tau_{432})&=27,\\
I_3(\tau_{21},\tau_{421},\tau_{431},\tau_{432},\tau_{432})&=30,&
I_3(\tau_4,\tau_{41},\tau_{432},\tau_{432},\tau_{432})&=16,\\
I_3(\tau_4,\tau_{32},\tau_{432},\tau_{432},\tau_{432})&=18,&
I_3(\tau_{31},\tau_{41},\tau_{432},\tau_{432},\tau_{432})&=34,\\
I_3(\tau_{31},\tau_{32},\tau_{432},\tau_{432},\tau_{432})&=54,&
I_3(\tau_{41},\tau_{41},\tau_{431},\tau_{432},\tau_{432})&=24,\\
I_3(\tau_{41},\tau_{32},\tau_{431},\tau_{432},\tau_{432})&=33,&
I_3(\tau_{32},\tau_{32},\tau_{431},\tau_{432},\tau_{432})&=54,\\
I_3(\tau_{431},\tau_{432},\tau_{432},\tau_{432})&=11.
\end{align*}
\end{small}
\caption{Degree 3, Case 10 numbers}
\label{table:d3c10}
\end{table}

\begin{table}
\addtolength{\abovecaptionskip}{-16pt}
\begin{small}
\setlength{\jot}{0pt}
\begin{align*}
I_4(\tau_2,\tau_2,\tau_2,\tau_{4321},\tau_{4321},\tau_{4321},\tau_{4321})&=64,&
I_4(\tau_2,\tau_2,\tau_3,\tau_{432},\tau_{4321},\tau_{4321},\tau_{4321})&=208,\\
I_4(\tau_2,\tau_2,\tau_{21},\tau_{432},\tau_{4321},\tau_{4321},\tau_{4321})&=160,&
I_4(\tau_2,\tau_3,\tau_3,\tau_{432},\tau_{432},\tau_{4321},\tau_{4321})&=576,\\
I_4(\tau_2,\tau_3,\tau_{21},\tau_{432},\tau_{432},\tau_{4321},\tau_{4321})&=420,&
I_4(\tau_2,\tau_3,\tau_{4321},\tau_{4321},\tau_{4321},\tau_{4321})&=8,\\
I_4(\tau_2,\tau_{21},\tau_{21},\tau_{432},\tau_{432},\tau_{4321},\tau_{4321})&=304,&
I_4(\tau_2,\tau_{21},\tau_{4321},\tau_{4321},\tau_{4321},\tau_{4321})&=8,\\
I_4(\tau_2,\tau_4,\tau_{432},\tau_{4321},\tau_{4321},\tau_{4321})&=12,&
I_4(\tau_2,\tau_{31},\tau_{432},\tau_{4321},\tau_{4321},\tau_{4321})&=38,\\
I_4(\tau_2,\tau_{41},\tau_{432},\tau_{432},\tau_{4321},\tau_{4321})&=62,&
I_4(\tau_2,\tau_{32},\tau_{432},\tau_{432},\tau_{4321},\tau_{4321})&=100,\\
I_4(\tau_3,\tau_3,\tau_{432},\tau_{4321},\tau_{4321},\tau_{4321})&=28,&
I_4(\tau_3,\tau_{21},\tau_{432},\tau_{4321},\tau_{4321},\tau_{4321})&=22,\\
I_4(\tau_3,\tau_4,\tau_{432},\tau_{432},\tau_{4321},\tau_{4321})&=36,&
I_4(\tau_3,\tau_{31},\tau_{432},\tau_{432},\tau_{4321},\tau_{4321})&=94,\\
I_4(\tau_{21},\tau_{21},\tau_{432},\tau_{4321},\tau_{4321},\tau_{4321})&=16,&
I_4(\tau_{21},\tau_4,\tau_{432},\tau_{432},\tau_{4321},\tau_{4321})&=26,\\
I_4(\tau_{21},\tau_{31},\tau_{432},\tau_{432},\tau_{4321},\tau_{4321})&=68,&
I_4(\tau_4,\tau_{4321},\tau_{4321},\tau_{4321},\tau_{4321})&=0,\\
I_4(\tau_{31},\tau_{4321},\tau_{4321},\tau_{4321},\tau_{4321})&=2,&
I_4(\tau_{41},\tau_{432},\tau_{4321},\tau_{4321},\tau_{4321})&=3,\\
I_4(\tau_{32},\tau_{432},\tau_{4321},\tau_{4321},\tau_{4321})&=6,&
I_4(\tau_{42},\tau_{432},\tau_{432},\tau_{4321},\tau_{4321})&=14,\\
I_4(\tau_{321},\tau_{432},\tau_{432},\tau_{4321},\tau_{4321})&=8.
\end{align*}
\end{small}
\caption{Degree 4, Case 1 numbers}
\label{table:d4c1}
\end{table}

\begin{table}
\addtolength{\abovecaptionskip}{-16pt}
\begin{small}
\setlength{\jot}{0pt}
\begin{align*}
I_4(\tau_3,\tau_3,\tau_3,\tau_{432},\tau_{432},\tau_{432},\tau_{4321})&=1488,&
I_4(\tau_3,\tau_3,\tau_{21},\tau_{432},\tau_{432},\tau_{432},\tau_{4321})&=1062,\\
I_4(\tau_3,\tau_{21},\tau_{21},\tau_{432},\tau_{432},\tau_{432},\tau_{4321})&=764,&
I_4(\tau_3,\tau_{41},\tau_{432},\tau_{432},\tau_{432},\tau_{4321})&=154,\\
I_4(\tau_3,\tau_{32},\tau_{432},\tau_{432},\tau_{432},\tau_{4321})&=232,&
I_4(\tau_{21},\tau_{21},\tau_{21},\tau_{432},\tau_{432},\tau_{432},\tau_{4321})&=552,\\
I_4(\tau_{21},\tau_{41},\tau_{432},\tau_{432},\tau_{432},\tau_{4321})&=108,&
I_4(\tau_{21},\tau_{32},\tau_{432},\tau_{432},\tau_{432},\tau_{4321})&=170,\\
I_4(\tau_{43},\tau_{432},\tau_{432},\tau_{432},\tau_{4321})&=26,&
I_4(\tau_{421},\tau_{432},\tau_{432},\tau_{432},\tau_{4321})&=29.
\end{align*}
\end{small}
\caption{Degree 4, Case 2 numbers}
\label{table:d4c2}
\end{table}

\begin{table}
\addtolength{\abovecaptionskip}{-16pt}
\begin{small}
\setlength{\jot}{0pt}
\begin{align*}
I_5(\tau_2,\tau_2,\tau_{4321},\ldots,\tau_{4321})&=125,&
I_5(\tau_2,\tau_3,\tau_{432},\tau_{4321},\ldots,\tau_{4321})&=250,\\
I_5(\tau_2,\tau_{21},\tau_{432},\tau_{4321},\ldots,\tau_{4321})&=175,&
I_5(\tau_3,\tau_3,\tau_{432},\tau_{432},\tau_{4321},\ldots,\tau_{4321})&=566,\\
I_5(\tau_3,\tau_{21},\tau_{432},\tau_{432},\tau_{4321},\ldots,\tau_{4321})&=403,&
I_5(\tau_3,\tau_{4321},\ldots,\tau_{4321})&=15,\\
I_5(\tau_{21},\tau_{21},\tau_{432},\tau_{432},\tau_{4321},\ldots,\tau_{4321})&=288,&
I_5(\tau_{21},\tau_{4321},\ldots,\tau_{4321})&=10,\\
I_5(\tau_4,\tau_{432},\tau_{4321},\ldots,\tau_{4321})&=14,&
I_5(\tau_{31},\tau_{432},\tau_{4321},\ldots,\tau_{4321})&=33,\\
I_5(\tau_{41},\tau_{432},\tau_{432},\tau_{4321},\ldots,\tau_{4321})&=50,&
I_5(\tau_{32},\tau_{432},\tau_{432},\tau_{4321},\ldots,\tau_{4321})&=75,\\
I_6(\tau_2,\tau_{4321},\ldots,\tau_{4321})&=60,&
I_6(\tau_3,\tau_{432},\tau_{4321},\ldots,\tau_{4321})&=180,\\
I_6(\tau_{21},\tau_{432},\tau_{4321},\ldots,\tau_{4321})&=130,&
I_7(\tau_{4321},\ldots,\tau_{4321})&=71.
\end{align*}
\end{small}
\caption{Degree 5, 6, and 7 numbers (all Case 1)}
\label{table:d567}
\end{table}

\section*{Higher degree numbers}
The associativity relations also determine many higher-degree
Gromov-Witten numbers.
For instance, taking $d=3$ we may apply
\eqref{eqn:assocrel} with
$\lambda^{m-3}=1$, $\lambda^{m-2}=432$, $|\lambda^{m-1}|\geq 2$,
$\lambda^m=4321$ just as in Case 1 above, and obtain many
Gromov-Witten numbers $I_3(\ldots,\tau_{4321},\tau_{4321})$.
However, since we obtained only \emph{some} of the degree $2$
Gromov-Witten numbers in the previous section, we need to check that
the contributions with $d'=2$ or $d-d'=2$ involve only
degree $2$ Gromov-Witten numbers that have been determined.
This is checked on a case-by-case basis for each of the $35$ numbers listed in
Table \ref{table:d3c1} and each
corresponding application of \eqref{eqn:assocrel}.

\begin{exam}
To determine
$I_3(\tau_{421},\tau_{431},\tau_{4321},\tau_{4321})$ we read off from
\eqref{eqn:assocrel} with $d=3$, $m=5$, and
$(\lambda^1,\ldots,\lambda^5)=(431, 1, 432, 421, 4321)$, the identity
(cf.\ \eqref{eqn:somelinenumbers}, \eqref{eqn:someconicnumbers}):
\begin{align*}
&I_3(\tau_{421},\tau_{431},\tau_{4321},\tau_{4321}) \\
&\qquad
= I_1(\tau_1,\tau_{43},\tau_{4321})I_2(\tau_{21},\tau_{421},\tau_{431},\tau_{432})
+ I_1(\tau_1,\tau_{421},\tau_{4321})I_2(\tau_3,\tau_{421},\tau_{431},\tau_{432}) \\
&\qquad
+ I_2(\tau_1,\tau_{431},\tau_{431},\tau_{4321})I_1(\tau_2,\tau_{421},\tau_{432})
- I_1(\tau_1,\tau_{431},\tau_{432})I_2(\tau_2,\tau_{421},\tau_{431},\tau_{4321}) \\
&\qquad
- I_1(\tau_1,\tau_1,\tau_{431},\tau_{432})I_2(\tau_{421},\tau_{432},\tau_{4321})
- I_2(\tau_1,\tau_{431},\tau_{432},\tau_{432})I_1(\tau_1,\tau_{421},\tau_{4321}) \\
&\qquad
= 1\cdot 4 + 0\cdot 5 + 2\cdot 1 - 1\cdot 3 - 1\cdot 1 - 4\cdot 0=2.
\end{align*}
Alternatively $(\lambda^1,\ldots,\lambda^5)=(421, 1, 432, 431, 4321)$ yields
$I_3(\tau_{421},\tau_{431},\tau_{4321},\tau_{4321})=
1\cdot 4 + 0\cdot 5 + 0\cdot 2 + 2\cdot 1 - 1\cdot 3 - 1\cdot 1=2$.
\end{exam}

Reasoning as in Case 2 we obtain the Gromov-Witten numbers
$I_3(\ldots,\tau_{432},\tau_{4321})$ listed in
Table \ref{table:d3c2}.
Again it must be checked that each application of
\eqref{eqn:assocrel} requires only known conic numbers.

Similarly we reason as in Case 3 above to obtain the
$I_3(\ldots,\tau_{431},\tau_{4321})$ listed in
Table \ref{table:d3c3}.
We conclude our determination of $d=3$ numbers with the
$I_3(\ldots,\tau_{432},\tau_{432})$ listed in
Table \ref{table:d3c10}, for which the reasoning is as in Case 10.

An application of \eqref{eqn:assocrel} with
$d=4$ requires numbers of degrees $1$, $2$, and $3$.
It must be verified on a case-by-case basis that the required conic
and cubic numbers are among those already determined.
Tables \ref{table:d4c1} and \ref{table:d4c2} list the numbers
$I_4(\ldots,\tau_{4321},\tau_{4321})$, respectively
$I_4(\ldots,\tau_{432},\tau_{4321})$, which are treated by
reasoning as in Case 1 and Case 2, respectively.
For $d=5$, $6$, and $7$ we require only numbers with at least two
point conditions, hence we use the reasoning of Case 1.
Again it must be verified on a case-by-case basis that the required
numbers of every smaller degree are among those already determined.
The numbers are displayed in Table \ref{table:d567}.
The final number displayed is the desired
\[I_7(\tau_{4321},\tau_{4321},\tau_{4321},\tau_{4321},\tau_{4321},
\tau_{4321},\tau_{4321})=71,
\]
with the following enumerative interpretation.

\begin{propo}
\label{prop:71}
There are $71$ rational curves of degree $7$ through $7$ general points on
$\OGr(5,10)$.
\end{propo}

\begin{remar}
Semi-simplicity allows us to reconstruct the full quantum cohomology even 
without assuming that the ordinary cohomology is generated by 
$H^2$, see \cite{BM04} and \cite{Mas}.  (The case where the cohomology
is generated by $H^2$ is addressed in \cite{KM}.)  
This property was verified for orthogonal Grassmannians in \cite{CMP10}.
\end{remar}

\end{document}